\documentclass{daj}

\dajAUTHORdetails{%
  title = {A Decomposition of Multicorrelation Sequences  for Commuting Transformations along Primes}, 
  author = {Anh N. Le, Joel Moreira, and Florian K. Richter},
  plaintextauthor = {Anh N. Le, Joel Moreira, Florian K. Richter},
  runningtitle = {Commuting Multicorrelations along Primes}, 
  keywords = {multicorrelation sequences, nilsequences, commuting transformations},
}   

\dajEDITORdetails{%
   year={2021},
   number={4},
   received={6 February 2020},   
   published={7 June 2021},  
   doi={10.19086/da.22056},       
}   

\usepackage[english]{babel}
\usepackage{amsfonts,csquotes}
\usepackage{mathrsfs}
\usepackage{bbm}
\usepackage{latexsym}
\usepackage{math dots}
\usepackage{amssymb}
\usepackage{mathtools}

\usepackage{enumitem}
\setlist{nolistsep}
\usepackage{amsthm}
\usepackage[capitalize]{cleveref}

\numberwithin{equation}{section}

\newtheoremstyle{plain}{3mm}{3mm}{\slshape}{}{\bfseries}{.}{.5em}{}
\newtheoremstyle{definition}{2mm}{2mm}{}{}{\bfseries}{.}{.5em}{}
\theoremstyle{plain}
	
\newtheorem{theorem}{Theorem}[section]

\newtheorem{lemma}[theorem]{Lemma}

\newtheorem{proposition}[theorem]{Proposition}
\newtheorem{corollary}[theorem]{Corollary}
\theoremstyle{definition}
\newtheorem{definition}[theorem]{Definition}
\newtheorem{remark}[theorem]{Remark}

\newtheorem{question}{Question}

\theoremstyle{plain}
\newcounter{MainTheoremCounter}

\newtheorem{Maintheorem}[MainTheoremCounter]{Theorem}

\newcommand{\N}{\mathbb{N}}
\newcommand{\Z}{\mathbb{Z}}
\newcommand{\R}{\mathbb{R}}
\newcommand{\C}{\mathbb{C}}
\newcommand{\Q}{\mathbb{Q}}

\newcommand{\sign}{\operatorname{sign}}

\newcommand{\define}[1]{{\itshape #1}}

\renewcommand{\epsilon}{\varepsilon}
\renewcommand{\leq}{\leqslant}
\renewcommand{\geq}{\geqslant}
\renewcommand{\setminus}{\backslash}
\renewcommand{\Re}{{\rm Re}\,}

\renewcommand{\P}{\mathbb{P}}

\newcommand{\E}{\mathop{\mathbb{E}}}
\newcommand{\B}{\mathcal{B}}

\renewcommand{\d}{~\mathsf{d}}

\usepackage[normalem]{ulem}

\newcommand{\hk}[1]{{\left\vert\kern-0.25ex\left\vert\kern-0.25ex\left\vert #1\right\vert\kern-0.25ex\right\vert\kern-0.25ex\right\vert}}

\newcommand{\nil}{\operatorname{Nil}}
\newcommand{\nz}{{\mathbb{N}_0}} 

\begin{document}

\begin{frontmatter}[classification=text]

\title{A Decomposition of Multicorrelation Sequences  for Commuting Transformations along Primes} 

\author[anh]{Anh N. Le}
\author[joel]{Joel Moreira}
\author[florian]{Florian K. Richter\thanks{This author is supported by the National Science Foundation under grant number DMS~1901453.}}

\begin{abstract}
We study multicorrelation sequences arising from systems with commuting transformations.
Our main result is a refinement of a decomposition result of Frantzikinakis and it states that any multicorrelation sequences for commuting transformations can be decomposed, for every $\epsilon>0$, as the sum of a nilsequence $\phi(n)$ and a sequence $\omega(n)$ satisfying $$\lim_{N\to\infty}\frac{1}{N}\sum_{n=1}^N |\omega(n)|<\epsilon$$ and $$\lim_{N\to\infty}\frac{1}{|\P\cap [N]|}\sum_{p\in \P\cap [N]} |\omega(p)|<\epsilon.$$
\end{abstract}
\end{frontmatter}

\section{Introduction}
Given a measure preserving transformation $T$ on a probability space $(X,\mu)$ and functions $f_0,\dots,f_k\in L^\infty(X)$, the sequence
\begin{equation}
\label{eq:multicorrelation_seq_1}
\alpha(n)=\int_Xf_0\cdot T^nf_1\cdot T^{2n}f_2 {\cdot\ldots\cdot} T^{kn}f_k\d\mu
\end{equation}
is called a \define{multicorrelation sequence} \cite{Bergelson_Host_Kra05} or \define{multiple correlation sequence} \cite{Leibman10,Leibman15,Frantzikinakis15b,Host_Kra_18}.

Multicorrelation sequences play a central role in the theory of multiple recurrence and its connections to combinatorics and number theory.
The study of the structure of multicorrelation sequences was pioneered by Bergelson, Host and Kra \cite{Bergelson_Host_Kra05}, who showed that so-called \define{nilsequences} (see \cref{def_nilsequences} below) arise as the natural object governing their behavior.
More precisely, they proved that for any multicorrelation sequence $(\alpha(n))_{n \in N}$ defined as in \eqref{eq:multicorrelation_seq_1} for an ergodic system $(X, \mu, T)$, there exists a nilsequence $(\psi(n))_{n \in \N}$ such that\footnote{Given a sequence $a:\N\to\C$ we write $\lim_{N-M\to\infty}\frac1{N-M}\sum_{n=M}^Na(n)=L$ to denote that for every $\epsilon>0$ there exists $N_0$ such that $|\frac1{N-M}\sum_{n=M}^Na(n)-L|<\epsilon$ for all $M,N\in\N$ with $N-M>N_0$.}
\begin{equation}
\label{eqn_BHK_result}
\lim_{N-M\to\infty}\frac1{N-M}\sum_{n=M}^N\big|\alpha(n)-\psi(n)\big|=0.
\end{equation}
Their result was later generalized by Leibman in \cite{Leibman15} to multicorrelation sequences coming from non-ergodic systems and in \cite{Leibman10} to \define{polynomial multicorrelation sequences}, i.e., sequences of the form
\[
\alpha(n)=\int_X f_0\cdot T^{q_1(n)} f_1\cdot T^{q_2(n)}f_2 {\cdot\ldots\cdot} T^{q_k(n)}f_k\d\mu
\]
where $q_1,\ldots,q_k\in \Q[x]$ are polynomials satisfying $q_i(\N)\subset\Z$.\footnote{Naturally, if $T$ is non-invertible then the condition $q_i(\N)\subset\Z$ needs to be replaced by $q_i(\N)\subset\N$.}
Another strengthening was obtained in \cite{Le17, Tao_Teravainen_17}, where it was shown that additionally to \eqref{eqn_BHK_result} one has
\begin{equation}
\label{eqn_BHK_primes}
\lim_{N\to\infty}\frac1{|\P\cap [N]|}\sum_{p\in \P\cap [N]}\big|\alpha(q(p))-\psi(q(p))\big|=0,
\end{equation}
where $q \in \Z[x]$ is any non-constant polynomial, $\P$ is the set of primes and $[N]$ denotes the set $\{1,\dots,N\}$.
This shows that even along the subsequence of primes, the behavior of $(\alpha(n))_{n \in \N}$ is described by a nilsequence.

Given commuting measure preserving transformations $T_1,\dots,T_k$ on a probability space $(X,\B,\mu)$ and functions $f_0,\dots,f_k\in L^\infty(X)$, one can consider the more general expression
\begin{equation}
\label{eq:oct-3-5}
    \alpha(n)=\int_Xf_0\cdot T_1^n f_1\cdot T_2^n f_2\cdots T_k^n f_k\d\mu
\end{equation}
called a \define{multicorrelation sequence for commuting transformations}.
In \cite{Frantzikinakis15b} Frantzikinakis established an approximate version of the Bergelson-Host-Kra structure theorem in the case of commuting transformations, showing that for every multicorrelation sequence for commuting transformations $(\alpha(n))_{n \in \N}$ and every $\epsilon>0$ there exists a nilsequence $(\psi(n))_{n \in \N}$ such that
\begin{equation}\label{eq_frantzikinakis}
\lim_{N-M\to\infty}\frac1{N-M}\sum_{n=M}^N\big|\alpha(n)-\psi(n)\big| \leq \epsilon.
\end{equation}
It is still an open problem whether one can take $\epsilon=0$ in \eqref{eq_frantzikinakis} (see \cref{question_zeroepsilon} in \cref{sec:open_questions}).

Our main result is a generalization of Frantzikinakis's theorem along primes, answering affirmatively a question asked in \cite{Le17}:
\begin{Maintheorem}
\label{thm:decomposition_along_primes}
Given commuting measure preserving transformations $T_1,\dots,T_k$ on a probability space $(X,\B,\mu)$ and functions $f_0,\dots,f_k\in L^\infty(X)$, let
\begin{equation}
\label{eq:oct-14-1}
    \alpha(n)=\int_Xf_0\cdot T_1^n f_1\cdot T_2^n f_2\cdots T_k^n f_k\d\mu.
\end{equation}
Then for every $\epsilon > 0$, there exists a $k$-step nilsequence $(\psi(n))_{n \in \N}$ such that
    \[
        \lim_{N - M \to \infty} \frac{1}{N-M} \sum_{n=M}^{N-1} |\alpha(n) - \psi(n)| \leq \epsilon
    \]
    and
    \[
       \lim_{N\to\infty}\frac1{|\P\cap [N]|}\sum_{p\in \P\cap [N]}\big|\alpha(p)-\psi(p)\big| \leq \epsilon.
    \]
\end{Maintheorem}

The methods used in the proof of \cref{thm:decomposition_along_primes} are quite general and can be adapted to give the following enhancement.

\begin{Maintheorem}
\label{thm:decomposition_poly_correlation_along_primes}
Let $m, k \in \N$ and $q_{i,j} \in \Q[x]$ satisfying $q_{i,j}(\N) \subset \Z$ for all $1 \leq i \leq m, 1 \leq j \leq k$. Then there exists $\ell = \ell(m, k, \max_{i,j} \deg(q_{i,j}))$ such that the following happens: Let $\big(T_{i,j}\big)_{i\in[m],j\in[k]}$ be commuting measure preserving transformations on a probability space $(X,\B,\mu)$ and functions $f_0,\dots,f_k\in L^\infty(X)$. Define
\begin{equation}
\label{eq:nov-27-1}
    \alpha(n)=\int_X f_0 \cdot \prod_{i = 1}^{m} T_{i,1}^{q_{i,1}(n)} f_1 \cdot \prod_{i = 1}^m T_{i, 2}^{q_{i,2}(n)} f_2 \cdots \prod_{i=1}^m T_{i,k}^{q_{i,k}(n)} f_k\d\mu.
\end{equation}
Then for every $\epsilon > 0, r \in \N$ and $s \in \Z$, there exists an $\ell$-step nilsequence $(\psi(n))_{n \in \N}$ for which
    \[
        \lim_{N - M \to \infty} \frac{1}{N-M} \sum_{n=M}^{N-1} |\alpha(n) - \psi(n)| \leq \epsilon
    \]
    and
    \[
       \lim_{N\to\infty}\frac1{|\P\cap [N]|}\sum_{p\in \P\cap [N]}\big|\alpha(r p + s)-\psi(r p + s)\big| \leq \epsilon.
    \]
\end{Maintheorem}

It is natural to ask whether in \cref{thm:decomposition_poly_correlation_along_primes} one can take $\psi$ to be independent of $r$ and $s$. 
In fact we expect quite a lot more to be true (see \cref{question_zeroepsilon_prime}) but the methods in this paper do not seem to be strong enough to guarantee this strengthening.

\subsection{Proof strategy}
In this subsection we describe in broad strokes the main ideas behind our proof of \cref{thm:decomposition_along_primes}.
The reader is directed to \cref{sec:preliminaries} and Subsection \ref{sec:prelim_Furstenberg_system} for definitions.

To begin with, we call upon a standard trick from multiplicative number theory to replace the average $\E_{p\in \P\cap [N]}|\alpha(p)-\psi(p)|$ with the weighted average $\E_{n\in[N]}\Lambda(n)|\alpha(n)-\psi(n)|$, where $\Lambda$ is the classical von Mangoldt function.
Then, using the Gowers uniformity of the W-tricked von Mangoldt function \cite[Theorem 7.2]{Green_Tao10} and a transference principle introduced by Green and Tao \cite{Green_Tao08}, we can compare the weighted average $\E_{n\in[N]}\Lambda(n)|\alpha(n)-\psi(n)|$ with the unweighted average $\E_{n\in[N]}|\alpha(n)-\psi(n)|$.
A similar comparison was carried out by Tao and Ter\"av\"ainen in \cite{Tao_Teravainen_17} for multicorrelation sequences for single transformations.
In the case of multiple commuting transformations, however, a serious technical complication arises during this step.
To effectively relate the average $\E_{n\in[N]}\Lambda(n)|\alpha(n)-\psi(n)|$ to the average $\E_{n\in[N]}|\alpha(n)-\psi(n)|$ using the $W$-trick, one must actually compare $\E_{n\in[N]}\Lambda_{W,b}(n)|\alpha(Wn + b) - \psi(Wn+b)|$ with $\E_{n\in[N]}|\alpha(Wn + b) - \psi(Wn+b)|$ for large $W$ and uniformly over all $b\in[W]$ coprime to $W$.
As $W$ increases, it is more challenging to control the second type of average in the case of multiple commuting transformations than it is in the special case of single transformations, ultimately because \eqref{eq_frantzikinakis} is not available for $\epsilon=0$.

In order to overcome this issue, we need a variant of Frantzikinakis's theorem where we have better control on the nilsequences that appear.
To obtain this variant, we first found the following description of the Furstenberg system associated with a multicorrelation sequence.

\begin{theorem}
\label{thm:Furstenberg_sytem_correlations}
    Let $\alpha: \N \to \C$ be as defined in \eqref{eq:oct-3-5} and let $(X, T)$ be the topological Furstenberg system associated with $(\alpha(n))_{n \in \N}$.
    Then
    \begin{enumerate}[label=(\roman{enumi}),ref=(\roman{enumi}),leftmargin=*]
        \item $(X, T)$ is uniquely ergodic, and
        \item if $\mu$ is the unique $T$-invariant measure on $X$, then the system $(X, \mu, T)$ is measure theoretically isomorphic to an inverse limit of $k$-step nilsystems.
    \end{enumerate}
\end{theorem}

Using \cref{thm:Furstenberg_sytem_correlations} we can produce, for every multicorrelation sequence $(\alpha(n))_{n \in \N}$, a nilsequence that, in addition to satisfying \eqref{eq_frantzikinakis}, is composed of more elementary building blocks called \define{dual nilsequences}, whose anti-uniformity seminorm is easier to control.

\begin{theorem}
\label{thm:strengthening_Frantzikinakis_decomposition_1}
    Let $\alpha: \N \to \C$ be as defined in \eqref{eq:oct-3-5} with $\lVert f_i \rVert_{\infty} \leq 1$ for $0 \leq i \leq k$.
    Then for every $\epsilon > 0$, there exists a $k$-step nilsequence $(\psi(n))_{n \in \N}$ such that
    \begin{enumerate}[label=(\roman{enumi}),ref=(\roman{enumi}),leftmargin=*]
        \item $$\lim_{N-M \to \infty} \E_{n\in[M,N)} |\alpha(n) - \psi(n)| < \epsilon,\ \text{and}$$

        \item $(\psi(n))_{n \in \N}$ is a convex combination of finitely many dual nilsequences of the form $(D_{k+1} \phi(n))_{n \in \N}$ in which $(\phi(n))_{n \in \N}$ is a $k$-step nilsequence with $\lVert \phi \rVert_{U^{k+1}(\N)} \leq1$.
    \end{enumerate}
\end{theorem}

The enhanced control over the anti-uniformity of the dual nilsequences appearing in \cref{thm:strengthening_Frantzikinakis_decomposition_1} translates to an effective control on the size of the averages $\E_{n\in[N]}\big(\Lambda_{W,b}(n)-1\big)\big|\alpha(Wn + b) - \psi(Wn+b)\big|$, even for large $W$ and uniformly over $b$, which allows us to finish the proof of \cref{thm:decomposition_along_primes}.

\paragraph{Structure of the paper.}
We start by setting up some notation and background in \cref{sec:preliminaries}. \cref{sec:Furstenberg_systems} is used to provide some information on Furstenberg systems of bounded sequences and includes a proof of \cref{thm:Furstenberg_sytem_correlations}.
We prove \cref{thm:strengthening_Frantzikinakis_decomposition_1} in \cref{sec:enhancement} and Theorems \ref{thm:decomposition_along_primes} and \ref{thm:decomposition_poly_correlation_along_primes} in \cref{sec:decomposition_along_primes}.
Finally, in \cref{sec:open_questions}, we state some open questions.


\section{Preliminaries}

\label{sec:preliminaries}

\subsection{Gowers norms}
    For a finite set $A$ and a function $f: A \to \C$, define $\E_{x \in A} f(x) = \frac{1}{|A|} \sum_{x \in A} f(x)$.
    Throughout this paper we denote by $\Z_N$ the quotient $\Z/(N\Z)$.

    Given $k,N \in \N$ and a function $f : \Z_{N} \to \C$, we define the $k$-Gowers norm of $f$ on $\Z_N$, denoted by $\lVert f \rVert_{U^k(\Z_N)}$ to be
    \[
        \lVert f \rVert_{U^k(\Z_N)} = \left( \E_{n \in \Z_N} \E_{\underline{h} \in \Z_{N}^k} \prod_{\underline{\eta} \in \{0,1\}^k} \mathcal{C}^{|\underline{\eta}|} f(n + \underline{\eta} \cdot \underline{h}) \right)^{1/2^k},
    \]
     where $\mathcal{C}$ denotes complex conjugation and, for $\underline{\eta} \in \{0,1\}^k$ and $\underline{h} = (h_1, \ldots, h_k) \in [N]^k$, we let $|\underline{\eta}|$ be the number of $1$'s in $\underline{\eta}$ and $\underline{\eta} \cdot \underline{h}:=\eta_1 h_1 + \ldots + \eta_k h_k$.
    We also let $\{0,1\}^k_* := \{0,1\}^k \setminus \{(0, 0, \ldots, 0)\}$.

    Gowers \cite{Gowers01} proved that $\lVert \cdot \rVert_{U^k(\Z_N)}$ is a norm when $k \geq 2$, and that it satisfies the following analogue of the Cauchy-Schwarz Inequality (see {\cite[Lemma 3.8]{Gowers01} or \cite[Equation (B.12)]{Green_Tao10}}).
    \begin{proposition}[Cauchy-Schwarz-Gowers Inequality]
    \label{lem:Cauchy_Schwarz_Gowers_inequality}
    Let $k, N \in \N$. For every $\underline{\eta} \in \{0,1\}^k$, let $f_{\underline{\eta}}: \Z_N \to \C$. Then
    \[
        \left| \E_{n \in \Z_N} \E_{\underline{h} \in \Z_N^k} \prod_{\underline{\eta} \in \{0,1\}^k} f_{\underline{\eta}}(n + \underline{\eta} \cdot \underline{h}) \right| \leq \prod_{\underline{\eta} \in \{0,1\}^k} \lVert f_{\underline{\eta}} \rVert_{U^k(\Z_N)}.
    \]
\end{proposition}

\subsection{Uniformity and anti-uniform seminorms in $\N$}

\label{sec:anti-uniform seminorms}

Given a bounded sequence $\phi:\N\to\C$, we denote by ${\mathcal A}(\phi)$ the smallest closed sub-algebra of $\ell^\infty(\N)$ that contains $\phi$, and is invariant under the left-shift and under pointwise conjugation.
We say that $\phi$ is \define{uniquely ergodic} if
\begin{equation}
\label{eq:uniquely_ergodic}
\forall \psi\in{\mathcal A}(\phi)\qquad\lim_{N-M\to\infty}\frac1{N-M}\sum_{n=M}^N\psi(n) \text{ exists.}
\end{equation}
The choice for the term \emph{uniquely ergodic} will be clear after \cref{prop:uniquely_ergodic}.
It is well known that nilsequences are uniquely ergodic, and it follows from Walsh's ergodic theorem \cite{Walsh12} that multicorrelation sequences are uniquely ergodic as well.

If $\phi$ is uniquely ergodic, then the \define{$k$-uniformity seminorm of $\phi$} is defined as
\[
    \lVert \phi \rVert_{U^k(\mathbf{\N})} = \left( \lim_{H \to \infty} \E_{\underline{h} \in [H]^k} \lim_{N \to \infty} \E_{n \in [N]} \prod_{\underline{\eta} \in \{0,1\}^k} \mathcal{C}^{|\underline{\eta}|} \phi(n + \underline{\eta} \cdot \underline{h}) \right)^{1/2^k}.
\]

The sequence $\phi$ is called \emph{$k$-anti-uniform} if there exists $C>0$ such that for all uniquely ergodic sequences $b$,
\begin{equation}
\label{eq:oct-3-4}
    \limsup_{N-M \to \infty} \left| \E_{n \in [M, N)} \phi(n) b(n) \right| \leq C\|b\|_{U^k(\N)}.
\end{equation}

\begin{remark}
  This notion of anti-uniformity is weaker that the one defined in \cite{Frantzikinakis15b} since we only test sequences $b$ which are uniquely ergodic.
\end{remark}

The infimum of all $C$ that satisfy \eqref{eq:oct-3-4} is called the \define{$k$-anti-uniform seminorm of $\phi$} and is  denoted by $\lVert \phi \rVert_{U^k(\N)}^*$.

Frantzikinakis \cite{Frantzikinakis15b} showed that all $k$-multicorrelation sequences as defined in \eqref{eq:oct-14-1} with $f_0, f_1, \ldots, f_k$ bounded by $1$ are $(k+1)$-anti-uniform with a seminorm not exceeding $4$. 
By a careful computation, it can be shown that this anti-uniform norm is in fact not greater than $1$ (cf.\ \cite[Section 23.3.2]{Host_Kra_18}).
For more details on uniformity and anti-uniform seminorm, see \cite{Host_Kra09, Frantzikinakis15b, Frantzikinakis_Host_2018}.

\subsection{Nilsystems and nilsequences}
\label{sec:preliminary_nilsequences}
    Given $k \in \N$, a \emph{$k$-step nilmanifold} is a homogeneous space $G/\Gamma$ where $G$ is a $k$-step nilpotent Lie group and $\Gamma$ is a co-compact and discrete subgroup.
    The group $G$ acts naturally on $X:=G/\Gamma$ by left translations and the unique $G$-invariant measure on $X$ is denoted by $\mu_X$.
    Fix $g \in G$ and let $T_g:X\to X$ be the translation by $g$.
    The topological dynamical system $(X,T_g)$ is called a \emph{(topological) $k$-step nilsystem}. We also call the measure preserving system $(X,\mu_X,T_g)$ a \emph{(measurable) $k$-step nilsystem}.

\begin{definition}
\label{def_nilsequences}
Let $k\in\N$.
 A \emph{$k$-step nilsequence} is a sequence of the form $\phi(n)=F(T^nx)$ where $(X,T)$ is a $k$-step nilsystem, $x\in X$ and $F\in C(X)$ is a continuous function on $X$.
    If $F\in C^\infty(X)$ we say that $\phi$ is a \define{smooth $k$-step nilsequence}.
\end{definition}

A $k$-step nilsequence can be approximated uniformly by smooth $k$-step nilsequences.
 The family of $k$-step nilsequences forms a shift invariant sub-algebra of $\ell^{\infty}$ which is closed under complex conjugation.
 For more details on nilsystems see \cite{Auslander_Green_Hahn63}, and for details on nilsequences see \cite[Section 4.3.1]{Bergelson_Host_Kra05} or \cite[Section 11.3.2]{Host_Kra_18}.

\begin{remark}
There are a number of slightly different definitions for nilsequences used throughout the literature.
We follow the definition used in \cite{Host_Kra_18, Frantzikinakis_Host_2018}.
In \cite{Bergelson_Host_Kra05}, on the other hand, what we call a $k$-step nilsequence is called a \emph{basic $k$-step nilsequence}.
In \cite{Green_Tao10, Green_Tao12, Green_Tao_Ziegler12}, for the sequence $(F(g^n \cdot x))_{n \in \N}$ to be called a nilsequence, the function $F$ is required to be Lipschitz instead just being continuous.
\end{remark}

\subsection{Host-Kra seminorms and nilfactors}
Let $(X, \mu, T)$ be an ergodic measure preserving system and $F \in L^{\infty}(\mu)$.
The \emph{$k$-step Host-Kra seminorm} of $F$ is defined as
\[
    \hk{F}_k = \left( \lim_{H_1 \to \infty} \ldots \lim_{H_k \to \infty} \int_X \E_{\underline{h} \in [H_1] \times \ldots \times [H_k]} \prod_{\underline{\eta} \in \{0,1\}^k} T^{\underline{\eta} \cdot \underline{h}} (\mathcal{C}^{|\underline{\eta}|} F) \, d \mu \right)^{1/2^k}.
\]

An application of Holder's inequality shows that if $p$ is sufficiently large, depending on $k$, then the function $F\mapsto\hk{F}_k$ from $L^p(X)\to\R$ is continuous.
The existence of all the limits in the above definition was established in \cite{Host_Kra05}.
The \emph{$k$-step nilfactor of $(X, \mu, T)$} is the maximal factor that is measure theoretically isomorphic to an inverse limit of $k$-step nilsystems and is denoted by $Z_k(X)$. It is proved in \cite{Host_Kra05} that for all $F \in L^{\infty}(X)$ one has
\begin{equation}\label{eq_nilfactor}\E\big(F|Z_k(X)\big) = 0 \mbox{ if and only if } \hk{F}_{k+1} = 0.
\end{equation}

If $(X, \mu, T)$ is an inverse limit of ergodic nilsystems in the measure theoretical sense, then there exists a topological model for this system which is an inverse limit of nilsystems in the topological sense (see \cite[Section 13.3.1]{Host_Kra_18}).
In view of this fact, we henceforth do not distinguish between topological and measure theoretic inverse limits of ergodic nilsystems.

\subsection{Dual nilsequences}
\label{sec:prelim_dual_nilsequences}
    Let $(X,\mu,T)$ be a measure preserving system and let $F\in L^\infty(X)$.
    The \emph{dual function of $F$ of degree $k$} is denoted by $D_{k}F$ and is defined by
    \[
        D_{k} F(x) = \lim_{N \to \infty} \E_{\underline{h} \in \Phi_N} \prod_{\underline{\eta} \in \{0,1\}^k_*} \mathcal{C}^{|\underline{\eta}|} F(T^{\underline{\eta} \cdot \underline{h}} x)
    \]
    where $(\Phi_N)_{N \in \N}$ is any F{\o}lner sequence in $\Z^k$.
    The existence of the above limit in $L^2(X, \mu)$ is shown in \cite[Theorem 1.2]{Host_Kra05} (see also \cite[Theorem 28 in Section 8.4.6]{Host_Kra_18}).
    It is also shown in \cite[Theorem 27 in Section 12.3.4]{Host_Kra_18} that when $(X,T)$ is an ergodic nilsystem and $F$ is a continuous function on $X$, the convergence is uniform on $x \in X$.
    This implies that $D_{k} F$ is also a continuous function on $X$.
    It follows directly from the definition that $\int_X F \cdot D_k F \d \mu = \hk{F}_k^{2^k}$.

    Given a nilsequence $\phi\in\ell^\infty$, the \emph{degree $k$ dual sequence associated to $\phi$}, written as $D_{k} \phi$, is
    \[
        D_{k} \phi(n) = \lim_{N \to \infty} \E_{\underline{h} \in \Phi_N} \prod_{\underline{\eta} \in \{0,1\}^k_*} \mathcal{C}^{|\underline{\eta}|} \phi(\underline{\eta} \cdot \underline{h} + n)
    \]
    for any F{\o}lner sequence $(\Phi_N)_{N \in \N}$ in $\Z^k$.
    Writing  $\phi(n) = F(T^n x_0)$ for some continuous function $F$ on an ergodic nilsystem $(X, \mu, T)$, we see that $D_{k} \phi$ can be written as
    \[
        D_{k} \phi(n) = \lim_{N \to \infty} \E_{\underline{h} \in \Phi_N} \prod_{\underline{\eta} \in \{0,1\}^k_*} \mathcal{C}^{|\underline{\eta}|} F(T^{\underline{\eta} \cdot \underline{h}} T^n x_0) = D_{k} F(T^n x_0).
    \]
    Hence $D_{k} \phi$ is again a nilsequence arising from the same nilsystem as $\phi$, and in particular the limit defining $D_{k} \phi(n)$ exists for all $n\in\N$ and does not depend on the choice of the F\o lner sequence $(\Phi_N)_{N \in \N}$.
 We note for later use that in this case $\lVert \phi \rVert_{U^{k}(\N)} = \hk{F}_k$ for all $k \geq 2$ (\cite[Corollary 3.11]{Host_Kra09}).

    By writing $D_k F - D_k G$ as a telescoping sum, we obtain the following lemma which will be used later.

\begin{lemma}\label{lemma_dualLp}
  Let $(X,\mu,T)$ be a measure preserving system and let $F,G\in L^\infty(X)$ with $\|F\|_{L^\infty(X)}, \|G\|_{L^\infty(X)}\leq 1$.
  Then\footnote{Given quantities $A$ and $B$ which depend on $x_1,\dots,x_r,y_1,\dots,y_s$ we write $A\ll_{y_1,\ldots,y_s} B$ if there exists a constant $C>0$, that possibly depends on $y_1,\ldots,y_s$ but not on $x_1,\dots,x_r$, such that $A \leq C B$.} for every $k \in \N$,
  \[
    \lVert D_k F-D_k G \rVert_{L^1(X)} \ll_k \lVert F - G \rVert_{L^{1}(X)}.
  \]
\end{lemma}
\begin{proof}
    For $\underline{\eta} = (\eta_1, \ldots, \eta_k), \underline{\gamma} = (\gamma_1, \ldots, \gamma_k) \in \{0,1\}^k$, we write $\eta < \gamma$ if there exists $j \in \{1, \ldots, k\}$ such that $\eta_i = \gamma_i$ for all $i < j$ and $\eta_j < \gamma_j$. Let $(\Phi_N)_{N \in \N}$ be a F{\o}lner sequence in $\Z^k$. By definition, we have
    \begin{multline}
    \label{eq:dec-12-1}
        D_k F - D_k G = \lim_{N \to \infty} \E_{\underline{h} \in \Phi_N} \prod_{\underline{\eta} \in \{0,1\}^k_*} \mathcal{C}^{|\underline{\eta}|} T^{\underline{\eta} \cdot \underline{h}} F - \lim_{N \to \infty} \E_{\underline{h} \in \Phi_N} \prod_{\underline{\eta} \in \{0,1\}^k_*} \mathcal{C}^{|\underline{\eta}|} T^{\underline{\eta} \cdot \underline{h}} G = \\
        \sum_{\underline{\gamma} \in \{0,1\}^k_*} \lim_{N \to \infty} \E_{\underline{h} \in \Phi_N} \left( \prod_{\substack{\underline{\eta} \in \{0,1\}^k_* \\ \underline{\eta} < \underline{\gamma}}} \mathcal{C}^{|\underline{\eta}|} T^{\underline{\eta} \cdot \underline{h}} G \right) T^{\underline{\gamma} \cdot \underline{h}} (F - G) \left( \prod_{\substack{\underline{\eta} \in \{0,1\}^k_* \\ \underline{\eta} > \underline{\gamma}}} \mathcal{C}^{|\underline{\eta}|} T^{\underline{\eta} \cdot \underline{h}} F \right).
    \end{multline}
    Since $\lVert F \rVert_{L^{\infty}(X)}, \lVert G \rVert_{L^{\infty}(X)} \leq 1$, the $L^1$-norm of the right hand side of \eqref{eq:dec-12-1} is bounded above by $(2^k - 1) \lVert F - G \rVert_{L^{1}(X)}$.
\end{proof}

We will also need the following technical result about dual sequences.
\begin{lemma}\label{lemma_dualconverge}
  Let $\phi$ be a nilsequence and let $k\in\N$.
  Denote by $\underline{h}=(h_1,\dots,h_k)$.
  Then the sequence
  $$D_k^{[N]}\phi(n) = \E_{h_1,\dots,h_{k-1}\in[N]}\lim_{H\to\infty}\E_{h_k\in [H]}\prod_{\underline{\eta} \in \{0,1\}^k_*} \mathcal{C}^{|\underline{\eta}|} \phi(\underline{\eta} \cdot \underline{h} + n)$$
converges as $N\to\infty$ to $D_k\phi(n)$, and the convergence is uniform in $n$.
\end{lemma}
\begin{proof}
    Note that since $\phi$ is a nilsequence, for every $h_1, \ldots, h_{k-1}, n \in \N$, the limit
    \[
        \lim_{H \to \infty} \E_{h_k \in [H]} \prod_{\underline{\eta} \in \{0,1\}^k_*} \mathcal{C}^{|\underline{\eta}|} \phi(\underline{\eta} \cdot \underline{h} + n)
    \]
    exists. By contradiction, assume $D_k^{[N]} \phi$ does not converges uniformly to $D_k \phi$. Thus there exist $\epsilon > 0$ and arbitrarily large $N$ such that
    \[
        \left| D_k^{[N]} \phi(n) - D_k \phi(n) \right| > \epsilon
    \]
    for some $n \in \N$.
    It follows that there exists arbitrarily large $N$ and $H$ such that
    \[
        \left| \E_{h_1, \ldots, h_{k-1} \in [N]} \E_{h \in [H]} \prod_{\underline{\eta} \in \{0,1\}^k_*} \mathcal{C}^{|\underline{\eta}|} \phi(\underline{\eta} \cdot \underline{h} + n) - D_k \phi(n) \right| > \epsilon.
    \]
    But this contradicts the fact that for any F{\o}lner sequence $(\Phi_N)_{N \in \N}$ the limit
    \[
        \lim_{N \to \infty} \E_{\underline{h} \in \Phi_N} \prod_{\underline{\eta} \in \{0,1\}^k_*} \mathcal{C}^{|\underline{\eta}|} \phi(\underline{\eta} \cdot \underline{h} + n)
    \]
    converges uniformly to $D_k \phi$.
\end{proof}
\subsection{Primes and the Transference Principle}

The \emph{modified von Mangoldt function} $\Lambda': \N \to \R$ is defined as
\begin{equation*}
    \Lambda'(m) = \begin{cases} \log m \, \mbox{ if } m \in \mathbb{P} \\
    0 \, \mbox{ otherwise.}
    \end{cases}
\end{equation*}

The following lemma is a well known corollary of the prime number theorem.
For a proof see, for example, \cite{Frantzikinakis_Host_Kra07}.
\begin{lemma}
\label{lem:von-Mangoldt}
    Let $b: \N \to \C$ be a bounded sequence. Then
    \[
        \lim_{N \to \infty} \left|\E_{n \in [N]} \Lambda'(n) b(n) - \E_{p \in \P \cap [N]} b(p)  \right| = 0
    \]
\end{lemma}

Let $W \in \N$ be a number of the form $W = \prod_{p \in \mathbb{P},\, p < w} p$ for some $w \in \N$.
For $b \in [W]$ coprime to $W$, define the \define{W-tricked von Mangoldt} function as
\[
    \Lambda_{W, b}(m) = \frac{\phi(W)}{W} \Lambda'(W m + b)
\]
The $W$-tricked von Mangoldt function was first introduced by Green and Tao \cite{Green_Tao08}. We will make use of the following theorem, which follows from combining \cite[Proposition 6.4]{Green_Tao10} (or \cite[Proposition 9.1]{Green_Tao08}) and \cite[Proposition 10.3]{Green_Tao10}.

\begin{theorem}[Transference Principle]
\label{thm:dense_model_theorem_original}
    Let $k \geq 1$. Then there exist constants $C = C(k) > 10$ and $M = M(k)$ such that the following happens: Let $\epsilon > 0$, let $w: \N \to \R^+$ be any function with $w(N) \leq 1/2 \log \log N$, and define $W = W(N) = \prod_{p \in \P, p < w(N)} p$.
    Then there exists $N_0=N_0(k, \epsilon, w)$ such that for all $N\geq N_0$ and all $N' \in [C N, 2C N]$ we can decompose any function $g: \Z_{N'} \to \C$ satisfying $|g(n)| \leq \Lambda_{W, b}(n) \cdot 1_{[N/4, 3N/4]}(n)$ for some $b\in[W]$ coprime to $W$ and for all $n \in \Z_{N'}$ as $g = g_1 + g_2$ in such a way that
    \begin{enumerate}
        \item $|g_1(n)| \leq M$ for all $n \in \Z_{N'}$
        \item $\lVert g_2 \rVert_{U^{k+1}(\Z_{N'})} \leq \epsilon$
        \item and $g_1, g_2$ are supported on $[N]$.
    \end{enumerate}
\end{theorem}
\begin{remark}
    In \cite[Proposition 10.3]{Green_Tao10}, the function $g$ takes real values instead of complex values as in \cref{thm:dense_model_theorem_original}.
    However, by decomposing $g$ into its real and imaginary parts, it follows that \cite[Proposition 10.3]{Green_Tao10} also holds for complex valued functions.

    Moreover, it is concluded in \cite[Proposition 10.3]{Green_Tao10} that if $g$ is supported on $[-N, N]$ then we can arrange the matters so that $g_1$ and $g_2$ are supported on $[-2N, 2N]$. But by exact the same proof, we have our version stated above. More specifically, we can write $g(n) = g(n) \psi(n)$ where $\psi: Z_{N'} \to [0,1]$ equals to $1$ on $[N/4, 3N/4]$, vanishes outside of $[1, N]$ and interpolates smoothly in the range $[1, N/4] \cup [3N/4, N]$. Then if $g = g_1 + g_2$ is the previous decomposition, upon multiplying by $\psi$, we have $g = g_1 \psi + g_2 \psi$. By choosing $\psi$ carefully and with the same argument as in the proof of \cite[Proposition 10.3]{Green_Tao10}, $g_1 \psi$ and $g_2 \psi$ still enjoy the same conclusion as $g_1$ and $g_2$.
\end{remark}

In the previous theorem, $w$ can be taken to be any sufficiently slow growing function of $N$.
Hence, by fixing $\epsilon$, we can take $w$ to be independent of $N$ as in the following corollary.

\begin{corollary}
\label{thm:dense_model_theorem}
    Let $k \geq 1$ and $\epsilon > 0$.
    Then there exist integers $C=C(k)$, $M = M(k)$, $w = w(k,\epsilon)>0$, and $N_0=N_0(k,\epsilon)$ such that the following holds:
    For all $N\geq N_0$, if $N' = CN$ and $W = \prod_{p \in \mathbb{P}, p < w} p$, then any function $g:\Z_{N'} \to \C$ satisfying
    \begin{equation}
    \label{eq:oct-12-1}
        |g(n)| \leq \Lambda_{W,b}(n) \cdot 1_{[N/4,  3N/4]}(n), \qquad\forall n \in \Z_{N'},
    \end{equation}
    can be decomposed as $g = g_1 + g_2$ on $\Z_{N'}$ in such a way that
    \begin{enumerate}
        \item  $|g_1(n)| \leq M \mbox{ for all } n \in \Z_{N'}$
        \item $\lVert g_2 \rVert_{U^{k}(\Z_{N'})} \leq \epsilon$
        \item and $g_1, g_2$ are supported on $[N]$.
    \end{enumerate}
\end{corollary}
\begin{proof}
    Let $C$ and $M$ be as in the conclusion of \cref{thm:dense_model_theorem_original}. By contradiction, assume \cref{thm:dense_model_theorem} is not true. Then there exists an $\epsilon > 0$ and increasing sequences $(w_h)_{h \in \N}$, $(N_h)_{n \in \N}$, and a sequence of functions $(g_h)_{h \in \N}$ such that $w_h \leq 1/2 \log \log N_h$ for all $h \in \N$ and $g_h: \Z_{N'} \to \R$ satisfying \eqref{eq:oct-12-1} but can not be decomposed as stated.

    Define a function $w: \N \to \R^+$ by $w(N) = N_h$ if $N_h \leq N < N_{h+1}$. Then $w$ is a non-decreasing sequence with $w(N) \to \infty$ as $N \to \infty$ and $w(N) \leq 1/2 \log \log N$ for all $N \in \N$. Now $w$ and $g_h$ satisfy all of the hypothesis of \cref{thm:dense_model_theorem_original}, but do not satisfy its conclusion. This is a contradiction.
\end{proof}

\section{Approximate nilsequences and their Furstenberg systems}

\label{sec:Furstenberg_systems}

Following the terminology introduced in \cite{Frantzikinakis15b}, we say that a bounded sequence $\alpha:\N\to\C$ is an \emph{approximate $k$-step nilsequence} if for every $\epsilon > 0$, there exists a $k$-step nilsequence $\psi$ such that
\[
    \limsup_{N - M \to \infty} \E_{n \in [M, N)} |\alpha(n) - \psi(n)| < \epsilon.
\]

The following simple lemma will be useful in the sequel and it follows immediately from the fact that the set of $k$-step nilsequences forms a shift-invariant algebra (see \cite[Section 3.1.1]{Host_Kra_18}).
\begin{lemma}\label{lemma_approxNilAlgebra}
  The collection of all approximate $k$-step nilsequences forms a shift invariant algebra.
\end{lemma}

Frantzikinakis' main result in \cite{Frantzikinakis15b} states that a $k$-multicorrelation sequence for commuting transformations is an approximate $k$-step nilsequence.
The proof consists of characterizing approximate nilsequences as precisely those sequences which are both regular and anti-uniform.
For our purposes, we will need a strengthening of Frantzikinakis' characterization of approximate nilsequences, described by \cref{thm:strengthening_Frantzikinakis_decomposition} below.
To formulate and prove this strengthening, we need to invoke the notion of a Furstenberg system of a sequence.

\subsection{Furstenberg system of a bounded sequence}
\label{sec:prelim_Furstenberg_system}

We denote by $\nz$ the set of non-negative integers.
Given a bounded sequence $\alpha\colon\N\to\C$, we define its \emph{(topological) Furstenberg system} to be the pointed topological system $(X,T,x)$ defined as follows.
Let $K\subset\C$ be a compact set with $\alpha(\N)\subset K$ and endow the product $K^\nz$ with the product topology.
Let $x\in K^\nz$ be a point with $x_n=\alpha(n)$ for every $n\in\N$. Let $T:K^\nz\to K^\nz$ be the left shift and let $X:=\overline{\{T^nx:n\in\nz\}}$ be the orbit closure of $x$.
Observe that $\alpha$ can be recovered from its Furstenberg system as $\alpha(n)=F(T^nx)$ where $F:X\to\C$ is the projection onto the $0$-th coordinate.

The following observation will be used repeatedly.

\begin{lemma}
\label{lem:apr-4-1}
    Any $T$-invariant sub-algebra of $C(X)$ closed under conjugation and containing $F$ is dense in $C(X)$.
\end{lemma}
\begin{proof}
    The lemma follows from the Stone-Weierstrass theorem combined with the observation that the set $\{T^nF:n\in\nz\}$ separates points in $X$.
\end{proof}

Several properties of a sequence are encoded in its Furstenberg system.
For instance, the Furstenberg system of $\phi$ is uniquely ergodic if and only if $\phi$ is uniquely ergodic (see \cref{prop:uniquely_ergodic}).
In \cref{prop:oct-3-1} we also show that the Furstenberg system of a bounded sequence $\phi$ is a minimal nilsystem if and only if $\phi$ is a nilsequence.
\begin{proposition}
\label{prop:uniquely_ergodic}
    Let $\phi\colon \N \to \C$ be a bounded sequence. Then the Furstenberg system of $\phi$ is uniquely ergodic if and only if $\phi$ is uniquely ergodic (i.e., satisfies \eqref{eq:uniquely_ergodic}).
\end{proposition}
\begin{proof}
    Let $(X,T,x)$ be the Furstenberg system associated to $\phi$ and let $F\in C(X)$ be such that $\phi(n)=F(T^nx)$ for every $n\in\N$.
First assume that $\phi$ is uniquely ergodic.
Let $\mu$ be an ergodic invariant measure on $(X,T)$.
Since $x$ has a dense orbit, \cite[Proposition 3.9]{Furstenberg81} implies that it is quasi-generic for $\mu$, in the sense that there exists a sequence $(I_N)_{N\in\N}$ of intervals in $\N$ whose lengths tend to infinity and such that
$$\int_X H\d\mu=\lim_{N\to\infty}\frac1{|I_N|}\sum_{n\in I_N}H(T^nx)\qquad\forall~H\in C(X).$$
For every function $H:X\to\C$ which belongs to the $T$-invariant and conjugation invariant algebra generated by $F$, the sequence $\psi(n):=H(T^nx)$ belongs to the algebra ${\mathcal A}(\phi)$
(which was defined at the beginning of \cref{sec:anti-uniform seminorms}).
Using the fact that $\phi$ is uniquely ergodic it follows that
$$\int_XH\d\mu=\lim_{N-M\to\infty}\frac1{N-M}\sum_{n=M}^N\psi(n),$$
and in particular this quantity depends only on $\phi$ and $H$, but not in the choice of $\mu$.
Invoking \cref{lem:apr-4-1}, this implies that the integral $\int_XH\d\mu$ does not depend on the choice of $\mu$, for every $H$ in a dense subset of $C(X)$.
Finally, in view of the Riesz representation theorem, we conclude that there is a unique invariant measure $\mu$ on $(X,T)$.

Conversely, if the Furstenberg system $(X,T,x)$ is uniquely ergodic, then for every $H\in C(X)$ the limit
$$\lim_{N-M\to\infty}\frac1{N-M}\sum_{n=M}^NH(T^nx).$$
exists.
Since for every $\psi\in{\mathcal A}(\phi)$ there exists $H\in C(X)$ such that $\psi(n)=H(T^nx)$, we conclude that $\phi$ is uniquely ergodic.
\end{proof}

We collect a few lemmas of dynamical nature which will be invoked in the proofs of Theorems \ref{thm:strengthening_Frantzikinakis_decomposition} and \ref{thm:Furstenberg_sytem_correlations}.

\begin{lemma}
\label{lem:two-norms-equal}
Let $\alpha$ be a uniquely ergodic sequence and let $(X,T,x)$ be its Furstenberg system, with unique invariant measure $\mu$.
Let $G\in C(X)$ and define $b(n)\coloneqq G(T^nx)$ for all $n \in \N$.
    Then for each $k \in \N$, we have $\hk{G}_{k} = \lVert b \rVert_{U^{k}(\N)}$.
\end{lemma}
\begin{proof}

    By definition of the $\hk{\cdot}_{k}$-seminorm, we have
    \[
        \hk{G}_{k}^{2^{k}} = \lim_{H\to\infty}\E_{\underline{h}\in[H]^{k}} \int_X \prod_{\underline{\epsilon} \in \{0,1\}^{k}} C^{|\underline{\epsilon|}} T^{\underline{\epsilon} \cdot \underline{h}} G \, d \mu.
    \]
    Since $\mu$ is uniquely ergodic, $x\in X$ is a generic point, so we can write the right hand side of the previous equality as
    \[
        \lim_{H\to\infty}\lim_{N\to\infty}\E_{\underline{h}\in[H]^{k}}\E_{n\in[N]} \prod_{\underline{\epsilon} \in \{0,1\}^{k}} C^{|\underline{\epsilon}|} b(n + \underline{\epsilon} \cdot \underline{h}).
    \]
    This last expression equals $\lVert b \rVert_{U^{k}(\N)}^{2^{k}}$, proving the claim.
\end{proof}

\begin{proposition}[see {\cite[Proposition 6.1]{Host_Kra09}} or {\cite[page 387]{Host_Kra_18}}]
\label{prop:generic_for_joining}
    Suppose $(X, T)$ is a topological system, $x\in X$ is a transitive point (i.e.\, a point with a dense orbit), and $\mu$ is an invariant ergodic measure on $X$.
    Let $(Y,S)$ be a distal topological system, $\nu$ be an invariant measure on $Y$ and let $\pi: (X, \mu, T) \to (Y, \nu, S)$ be a measure theoretic factor map.
    Then there exists a point $y \in Y$ and a sequence of intervals $(I_N)_{N \in \N}$ such that
    \begin{equation}
        \lim_{N \to \infty} \E_{n \in I_N} f(T^n x) g(S^n y) = \int_X f \cdot g \circ \pi \, d\mu.
    \end{equation}
    for all $f \in C(X)$ and $g \in C(Y)$.
\end{proposition}
\begin{corollary}
    \label{cor:generic_for_joining}
    Let the set-up be as in \cref{prop:generic_for_joining}. Then for any $P: \C^2\to\C$ continuous, we have
    \[
        \lim_{N \to \infty} \E_{n \in I_N} P(f(T^n x_0), g(S^n y_0)) = \int_X P(f, g \circ \pi) \, d\mu.
    \]
\end{corollary}
\begin{proof}
    It is easy to see the conclusion is true in the case $P(z_1,z_2)$ is polynomial on $z_1,z_2,\overline{z_1},\overline{z_2}$.
    Then by the Stone-Weierstrass theorem, the conclusion is true for an arbitrary continuous function $P$.
\end{proof}

We will also need the following proposition.
\begin{proposition}
    \label{prop:regular_joining}
    Let $\phi$ and $\psi$ be approximate nilsequences.
    Then for every continuous $P: \C^2 \to\C$, the following uniform Ces\`aro limit exists:
    \[
        \lim_{N - M \to \infty} \E_{n \in [M, N)}
P(\phi(n), \psi(n)).
    \]
\end{proposition}
\begin{proof}
  Using the Stone-Weierstrass theorem we may assume that $P(z_1,z_2)$ is a polynomial on $z_1,z_2,\overline{z_1},\overline{z_2}$.
  The result now follows from \cref{lemma_approxNilAlgebra} and the fact that the uniform Ces\`aro limit of a nilsequence exists.
\end{proof}

\subsection{The Furstenberg system of approximate nilsequences}
The following proposition reveals a new characterization of approximate nilsequences.
Even though this proposition is not needed in the sequel, it helps to put in perspective Theorems \ref{thm:Furstenberg_sytem_correlations} and \ref{thm:Furstenberg_system_correlations_general}.
\begin{proposition}
\label{prop:oct-3-1}
    A bounded sequence is a $k$-step nilsequence if and only if the associated topological Furstenberg system is isomorphic to a minimal $k$-step nilsystem.
\end{proposition}
\begin{proof}
    Let $\phi$ be a bounded sequence.
    First, assume its Furstenberg system $(X, T)$ is isomorphic to a minimal $k$-step nilsystem $(\tilde X,\tilde T)$ and let $\rho:\tilde X\to X$ be the isomorphism. By the definition of a Furstenberg system, there exists $F \in C(X)$ and $x \in X$ such that $\phi(n) = F(T^n x)$.
    Letting $\tilde x=\rho^{-1}x$ it follows that $\phi(n)=(F\circ\rho)(\tilde T^n\tilde x)$, and hence $\phi$ is a $k$-step nilsequence.

    Conversely, assume $\phi$ is a $k$-step nilsequence.
    Hence there exist a $k$-step nilsystem $(Y,S)$, a function $G \in C(Y)$ and $y_0 \in Y$ such that $\phi(n) = G(S^n y_0)$ for all $n \in \N_0$.
    By restricting to the orbit closure of $y_0$, we can assume that $(Y, S)$ is transitive (hence minimal and unique ergodic).
    Let $(X, T, x)$ be the pointed Furstenberg system of $\phi$.

    We claim that $(X,T)$ is a factor of $(Y,S)$.
    Since a factor of a minimal $k$-step nilsystem is again a minimal $k$-step nilsystem (see \cite{Parry_1973} or \cite[Chapter 13, Theorem 11]{Host_Kra_18}) this will conclude the proof.

    We prove the claim by explicitly constructing a factor map $\pi:Y\to X$.
    Given $y\in Y$ let $\pi(y)=\big(G(y),G(Sy),G(S^2y),\dots\big)$.
    Observe that $\pi(S^n y_0) = T^n x$ for all $n \in \N_0$.
    Since $(Y,S)$ is transitive, for every $y\in Y$ there is a sequence $(n_i)_{i \in \N}$ such that $S^{n_i} y_0 \to y$.
    Since $G$ is continuous, the sequence $T^{n_i} x=(G(S^{n_i}) y_0, G(S^{n_i + 1}) y_0, \ldots)$ converges to $\pi(y)$, showing that $\pi(y)\in X$.
    A similar argument shows that $\pi$ is continuous and surjective, and hence a factor map.
\end{proof}

Here is the main theorem of this section.
\begin{theorem}
\label{thm:Furstenberg_system_correlations_general}
    Let $\alpha: \N \to \C$ be an approximate $k$-step nilsequence and $(X, T)$ be the topological Furstenberg system associated to $\alpha$. Then
    \begin{enumerate}[label=(\roman{enumi}),ref=(\roman{enumi}),leftmargin=*]
        \item $(X, T)$ is uniquely ergodic, and
        \item If $\mu$ is the unique $T$-invariant measure on $X$, then the system $(X, \mu, T)$ is measure theoretically isomorphic to an inverse limit of $k$-step nilsystems.
    \end{enumerate}
\end{theorem}
\begin{proof}

By combining \cref{prop:uniquely_ergodic} with \cref{lemma_approxNilAlgebra} and the fact that the uniform Ces\`aro average of an approximate nilsequence exists it follows that the Furstenberg system $(X,T,x)$ of $\alpha$ is uniquely ergodic.
Let $F\in C(X)$ be the function which generates $\alpha$ in the sense that $\alpha(n)=F(T^nx)$ and let $\mu$ be the unique invariant measure on $X$.

Let $Z_k$ be the $k$-step nilfactor of $(X,\mu,T)$. Let $\pi:X\to Z_k$ be the factor map and by abuse of notation, identify $L^{\infty}(Z_k)$ with $\big\{f\circ\pi:f\in L^\infty(Z_k)\big\}\subset L^\infty(X)$.
We claim that $F\in L^{\infty}(Z_k)$.
Assuming the claim for now, since $L^{\infty}(Z_k)$ is a closed $T$-invariant and conjugation invariant algebra, we have from \cref{lem:apr-4-1} that $C(X)\subset L^{\infty}(Z_k)$.
This in turn implies that $L^{\infty}(Z_k) = L^\infty(X)$ and hence that $\pi$ is an isomorphism, finishing the proof.

We are left to prove the claim that $F \in L^{\infty}(Z_k)$.
Since $\alpha$ is an approximate $k$-step nilsequence, for every $\epsilon > 0$, there exists a smooth $k$-step nilsequence $\psi_\epsilon$ such that
\begin{equation}
\label{eq:oct-3-2}
    \lim_{N-M \to \infty} \E_{n \in [M, N)} |\alpha(n) - \psi_\epsilon(n)|^2 < \epsilon^2.
\end{equation}
By \cite[Corollary 2.15]{Host_Kra09}, smooth $k$-step nilsequences are $(k+1)$-anti-uniform.
In other words, $\lVert \psi_\epsilon \rVert_{U^{k+1}(\N)}^*<\infty$. 

\begin{lemma}
\label{lem:apr-4-2}
For every $G \in L^\infty(X)$ and $\epsilon>0$,
    \begin{equation}\label{eq_lemma_thm_nest}
    \left| \int_X F \cdot G \, d \mu \right| \leq \lVert \psi_\epsilon \rVert_{U^{k+1}(\N)}^* \cdot\hk{G}_{k+1} + \epsilon \|G\|_{L^2},
    \end{equation}
    where $\psi_\epsilon$ is a smooth nilsequence satisfying \eqref{eq:oct-3-2}. 
\end{lemma}
\begin{proof}
\renewcommand{\qedsymbol}{$\blacklozenge$}
If $p$ is large enough depending on $k$, then both sides of \eqref{eq_lemma_thm_nest} depend continuously on $G$ with respect to the $L^p$ norm.
Since every function in $L^\infty(X)$ can be approximated by continuous functions in the $L^p$ norm, it suffices to prove the statement in the special case when $G$ is continuous.

    Let $b(n)=G(T^nx)$ and note that since $(X,T)$ is uniquely ergodic, $b$ is also uniquely ergodic.
    Using the fact that $x$ is generic for $\mu$, we have
    $$\left| \int_X F \cdot G \, d \mu \right|
      =
        \lim_{N\to\infty}\left|\E_{n \in [N]} \alpha(n) b(n) \right|.$$
    Next, using the triangle inequality and the Cauchy-Schwarz inequality we have
     $$\left|\E_{n \in [N]} \alpha(n) b(n) \right|\leq
        \left|\E_{n \in [N]}\psi_\epsilon(n) b(n) \right|
        + \sqrt{\E_{n\in[N]}|\alpha(n)-\psi_\epsilon(n)|^2\cdot\E_{n\in[N]}|b(n)|^2}.$$
    Finally, combining the above with
    \eqref{eq:oct-3-2} and the definition of anti-uniformity seminorms we conclude that
    \begin{eqnarray*}
      \left| \int_X F \cdot G \, d \mu \right|
      &\leq&
        \lVert \psi_\epsilon \rVert_{U^{k+1}(\N)}^* \cdot \lVert b \rVert_{U^{k+1}(\N)} + \epsilon\sqrt{\lim_{N\to\infty}\E_{n\in[N]}|b(n)|^2}
        \\&=&
        \lVert \psi_\epsilon \rVert_{U^{k+1}(\N)}^* \cdot \lVert b \rVert_{U^{k+1}(\N)} + \epsilon \lVert G \rVert_{L^2(X)}.
    \end{eqnarray*}
    \cref{lem:two-norms-equal} implies that $\lVert b \rVert_{U^{k+1}(\N)} = \hk{G}_{k+1}$, which finishes the proof of the lemma. 
\end{proof}

We are now ready to prove the claim that $F \in L^{\infty}(Z_k)$.
This is equivalent to the statement that $F$ is orthogonal to any $G\in L^\infty(X)$ satisfying $\hk{G}_{k+1} = 0$.
    Given such $G$, \cref{lem:apr-4-2} implies that
    \[
         \left| \int_X F \cdot G \, d \mu \right|
         \leq
         \lVert \psi_\epsilon \rVert_{U^{k+1}(\N)}^* \cdot\hk{G}_{k+1} + \epsilon \|G\|_{L^2}
         =
         \epsilon \|G\|_{L^2}.
    \]
    Since $\epsilon$ is arbitrary, we have $\int F \cdot G = 0$, and hence $F$ is indeed orthogonal to any $G$ satisfying $\hk{G}_{k+1} = 0$.
    This proves the claim that $F \in L^{\infty}(Z_k)$ and concludes the proof of \cref{thm:Furstenberg_system_correlations_general}.
\end{proof}

\begin{proof}[Proof of \cref{thm:Furstenberg_sytem_correlations}]
    By \cite{Frantzikinakis15b}, $k$-multicorrelation sequences are approximate $k$-step nilsequences. Hence \cref{thm:Furstenberg_sytem_correlations} follows from \cref{thm:Furstenberg_system_correlations_general}.
\end{proof}

\section{An enhancement of Frantzikinakis' decomposition}

\label{sec:enhancement}

{In this section we prove \cref{thm:strengthening_Frantzikinakis_decomposition_1}.
We will in fact establish the following more general result.}
\begin{theorem}
\label{thm:strengthening_Frantzikinakis_decomposition}
    Let $\alpha$ be an approximate $k$-step nilsequence with $\lVert \alpha \rVert_{U^{k+1}(\N)}^* \leq1$.
    Then, for every $\epsilon > 0$, there exists a $k$-step nilsequence $\psi$ such that
    \begin{enumerate}[label=(\roman{enumi}),ref=(\roman{enumi}),leftmargin=*]
        \item $$\lim_{N-M \to \infty} \E_{n \in [M, N)} |\alpha(n) - \psi(n)| < \epsilon$$
        \item $\psi$ is a convex combination of dual nilsequences of the form $D_{k+1} \phi$ with $\phi$ being some $k$-step nilsequence and $\lVert \phi \rVert_{U^{k+1}(\N)} \leq1$.
    \end{enumerate}
\end{theorem}

\begin{proof}[Proof of \cref{thm:strengthening_Frantzikinakis_decomposition_1} using \cref{thm:strengthening_Frantzikinakis_decomposition}]
    Let $\alpha$ be defined as in \cref{thm:strengthening_Frantzikinakis_decomposition_1}.
    By \cite{Frantzikinakis15b}, $\alpha$ is an approximate $k$-step nilsequence and by \cite[Section 23.3.2]{Host_Kra_18}, $\lVert \alpha \rVert_{U^{k+1}(\N)}^* \leq 1$ (see \cref{sec:anti-uniform seminorms}).
    Therefore, \cref{thm:strengthening_Frantzikinakis_decomposition_1} follows from \cref{thm:strengthening_Frantzikinakis_decomposition}.
\end{proof}

Before going into the proof of \cref{thm:strengthening_Frantzikinakis_decomposition}, we need a definition.
\begin{definition}
  Let $(X,\mu,T)$ be a measure preserving system. We denote by $\nil_k(X)$ the space of all functions on $X$ of the form $G\circ\pi$ where $\pi:(X,\mu,T)\to(Y,\nu,S)$ is a factor map of measure preserving systems, $(Y,\nu,S)$ is a $k$-step nilsystem and $G\in C(Y)$.
\end{definition}
The following lemma is inspired by \cite[Proposition 5.7]{Host_Kra09}.
\begin{lemma}
    \label{lem:antiuniform_in_convex_hull}
    Let $\alpha$ be an approximate $k$-step nilsequence with $\|\alpha\|_{U^{k+1}(\N)}^*\leq 1$.
    Let $(X, \mu, T)$ be its Furstenberg system and let $F$ be the continuous function associated to $\alpha$ as constructed in \cref{sec:prelim_Furstenberg_system}.
    Let
    $${\mathcal A}:=\Big\{D_{k+1} G: G\in \nil_k(X) \mbox{ and }  \hk{G}_{k+1} \leq 1\Big\}.$$
   Then $F$ is in the closed (with respect to the $L^1(\mu)$ topology) convex hull of ${\mathcal A}$.
\end{lemma}
\begin{proof}

    Let $\mathcal{K}$ denote the closed convex hull of ${\mathcal A}$.
    We assume, for the sake of a contradiction, that the conclusion of the lemma is false.
    Then by the Hahn-Banach Separation Theorem, there exist a real number $c$ and a function $H\in L^\infty(X)$ such that $\Re\langle H, F \rangle > c$ and $\Re\langle H, f \rangle \leq c$ for all $f \in \mathcal{K}$.
    After multiplying $H$ by an appropriate complex scalar if necessary (and changing $c$ accordingly) we can assume that $\hk{H}_{k+1}=1$ and $\langle H, F \rangle\in\R^{\geq0}$.

    Let $\epsilon>0$ be such that $\langle H, F \rangle > c+\epsilon$.
    We claim that
    \begin{equation}\label{eq_idk}
        \langle H, F \rangle \leq \hk{H}_{k+1}=1,
    \end{equation}
    whence $c<1-\epsilon$. 
    To justify \eqref{eq_idk}, first approximate $H$ by a continuous function $\tilde H$ in $L^{2^{k+1}}$, say $\lVert H - \tilde H \rVert_{L^{2^{k+1}}} \ll_F \delta$. 
    Let $b(n) = \tilde H(T^n x_0)$ for $n \in \N$. Then $(b(n))_{n \in \N}$ is a uniquely ergodic sequence and hence, using \cref{lem:two-norms-equal},
    \begin{multline*}
        |\langle H , F \rangle| \leq |\langle \tilde H, F \rangle| + \delta = |\E_n b(n) \alpha(n)| + \delta \leq \lVert b \rVert_{U^{k+1}(\N)} \lVert \alpha \rVert_{U^{k+1}(\N)}^* + \delta \\
        \leq \lVert b \rVert_{U^{k+1}} + \delta= \hk{\tilde H}_{k+1} + \delta \leq  \hk{H}_{k+1} + 2 \delta.
    \end{multline*}
    Letting $\delta\to0$ we obtain \eqref{eq_idk} as claimed.

    By \cref{thm:Furstenberg_system_correlations_general}, $(X, \mu, T)$ is an inverse limit of $k$-step nilsystems, which means that $\nil_{k}(X)$ is dense in $L^p(X)$ for every $p<\infty$.
    In particular, in view of \cref{lemma_dualLp}, we can find $H'\in\nil_{k}(X)$ such that $\hk{H'}_{k+1}=1$ and $\|D_{k+1}H-D_{k+1}H'\|_{L^1(X)}<\epsilon/\|H\|_{L^\infty(X)}$.
    In particular, $D_{k+1}H'$ is in $\mathcal{K}$ and so $\Re \langle H, D_{k+1}H'\rangle \leq c<1-\epsilon$. On the other hand,
    \[
        \Re\langle H, D_{k+1}H' \rangle
        \geq
        \Re \langle H, D_{k+1} H \rangle-\epsilon
        =
        \hk{H}_{k+1}-\epsilon
        =
        1-\epsilon
    \]
    providing the desired contradiction.
\end{proof}

\begin{proof}[Proof of \cref{thm:strengthening_Frantzikinakis_decomposition}]
Let $(X,T)$ be the Furstenberg system associated with $\alpha$.
In view of \cref{thm:Furstenberg_system_correlations_general}, this system is uniquely ergodic. Let $\mu$ be the unique invariant measure.

By \cref{lem:antiuniform_in_convex_hull}, there exists $t\in\N$ and, for each $i=1,\dots,t$, a $k$-step nilsystem $(Y_i,S_i)$ which is a factor of $(X,T)$ via a factor map $\pi_i$, and a function $G_i\in C(Y_i)$ such that $\lVert F - \tilde{G} \rVert_{L^1(\mu)} < \epsilon$ for some convex combination $\tilde G$ of the functions $D_{k+1} G_i\circ\pi_i$, $i=1,\dots,t$.
Since all the $(Y_i,S_i)$ are factors of $(X,T)$ and are $k$-step nilsystems, they are all factors of the maximal $k$-step nilfactor $Z_k(X)$. 
On the other hand, $Z_k(X)$ is an inverse limit of $k$-step nilsystems, so it follows that there exists a factor $(Y,S)$ of $(X,T)$ which contains each of $(Y_i,S_i)$ as a further factor.
From \cite[Corollary 5.3]{Host_Kra09} we deduce that $D_{k+1} G_i \in C(Y)$ for all $i$, and hence $\tilde G=G'\circ\pi$ where $\pi:X\to Y$ is the factor map and $G'\in C(Y)$.

Applying \cref{cor:generic_for_joining} to the function $P(z_1,z_2) = |z_1-z_2|$, there exists a sequence of intervals $(I_N)_{N \in \N}$ and  a point $y_0 \in Y$ such that
    \[
        \lim_{N \to \infty} \E_{n \in I_N} |F(T^n x_0) - G'(S^n y_0)| = \int_X |F - \tilde{G}| \, d\mu < \epsilon.
    \]
    Because $(F(T^n x_0))_{n \in \N}$ is an approximate nilsequence and $(G'(S^n y_0))_{n \in \N}$ is a $k$-step nilsequence, by \cref{prop:regular_joining}, the above average can be replaced by the uniform Ces\`aro average. Therefore,
    \[
        \lim_{N - M \to \infty} \E_{n \in [M, N)} |F(T^n x_0) -G'(S^n y_0)| < \epsilon.
    \]
    Let $\psi(n) = G'(S_i^n y_0)$.
    Then by \cref{sec:prelim_dual_nilsequences}, $\psi$ indeed is a convex combination of dual nilsequences of the form $D_{k+1} \phi$ with $\phi$ being a $k$-step nilsequence and $\lVert \phi \rVert_{U^{k+1}(\N)}\leq 1$.
\end{proof}

\section{Decomposition along the primes}
\label{sec:decomposition_along_primes}

The main of goal of this section is to prove \cref{thm:decomposition_along_primes}.
Before presenting its proof, we need some technical lemmas.

\begin{lemma}[{\cite[Chapter 22, Lemma 10]{Host_Kra_18}}]
\label{lem:sep-30-4}
    Let $k \in \N$ and $\phi$ be a uniquely ergodic sequence. Then
    \[
        \limsup_{N \to \infty} \lVert \phi 1_{[N]} \rVert_{U^k(\Z_N)} \ll_k \lVert \phi \rVert_{U^k(\N)}.
    \]
\end{lemma}

The following lemma is reminiscient of Cauchy-Schwarz-Gowers inequality. However, due to the involvement of both Gowers norm on cyclic groups and uniformity seminorm on $\N$, we need some technical modifications.

\begin{lemma}
\label{lem:g_2_and_dual_function}
    Let $k \in \N$ and $\phi: \N \to \C$ be a nilsequence.
    Let $g : \N \to \C$ and let $(N_l)_{l \in \N}$ be an increasing sequence of positive integers for which
    \[
        \limsup_{l \to \infty} \E_{n \in [N_l]} |g(n)| \leq 1.
    \]
    Then
    \[
        \limsup_{l \to \infty} \left| \E_{n \in [N_l]} g(n) D_k \phi(n) \right|
        \ll_k
        \limsup_{l \to \infty}\ \lVert g 1_{[N_l]} \rVert_{U^k(\Z_{k N_l})}\cdot \lVert \phi \rVert_{U^k(\N)}^{2^k-1}.
    \]
\end{lemma}

\begin{proof}
    As mentioned in \cref{sec:prelim_dual_nilsequences}, $D_k \phi$ is a nilsequence. In particular, it is bounded. Therefore,
    \[
        \limsup_{l \to \infty} \E_{n \in [N_l]} \left|g(n) D_k \phi(n) \right| < \infty.
    \]
    Passing to a subsequence of $(N_l)$ if necessary, we may assume that the limit
    \begin{equation}
    \label{eq:oct-1-2}
         A:=\lim_{l \to \infty} \E_{n \in [N_l]} g(n) D_k \phi(n)
    \end{equation}
   exists.

    Using \cref{lemma_dualconverge} we can write
    \begin{equation}\label{eq_proof_duallemma}
      A=\lim_{l\to\infty}\E_{n\in[N_l]}g(n)\E_{\underline{h}\in[N_l]^{k-1}}\lim_{H\to\infty}\E_{h_k\in[H]} \prod_{\underline{\eta} \in \{0,1\}^k_*} \mathcal{C}^{|\underline{\eta}|} \phi\big(n + \underline{\eta} \cdot (\underline{h},h_k)\big).
    \end{equation}
    We can rewrite the last limit as
    \begin{eqnarray*}
      &&\lim_{H\to\infty}\E_{h_k\in[H]} \prod_{\underline{\eta} \in \{0,1\}^k_*} \mathcal{C}^{|\underline{\eta}|} \phi\big(n + \underline{\eta} \cdot (\underline{h},h_k)\big)
      \\&=&
      \lim_{H\to\infty}\E_{h_k\in[H]}\prod_{\underline{\eta} \in \{0,1\}^{k-1}_*} \mathcal{C}^{|\underline{\eta}|} \phi\big(n + \underline{\eta} \cdot \underline{h}\big)
      \prod_{\underline{\eta} \in \{0,1\}^{k-1}} \mathcal{C}^{|\underline{\eta}|+1} \phi\big(n + h_k+ \underline{\eta} \cdot \underline{h}\big)
      \\&=&
      \prod_{\underline{\eta} \in \{0,1\}^{k-1}_*} \mathcal{C}^{|\underline{\eta}|} \phi\big(n + \underline{\eta} \cdot \underline{h}\big)\lim_{H\to\infty}\E_{h_k\in[H]-n}
      \prod_{\underline{\eta} \in \{0,1\}^{k-1}} \mathcal{C}^{|\underline{\eta}|+1} \phi\big(h_k+ \underline{\eta} \cdot \underline{h}\big)
      \\&=&
      \prod_{\underline{\eta} \in \{0,1\}^{k-1}_*} \mathcal{C}^{|\underline{\eta}|} \phi\big(n + \underline{\eta} \cdot \underline{h}\big)\lim_{H\to\infty}\E_{h_k\in[H]}
      \prod_{\underline{\eta} \in \{0,1\}^{k-1}} \mathcal{C}^{|\underline{\eta}|+1} \phi\big(h_k+ \underline{\eta} \cdot \underline{h}\big),
    \end{eqnarray*}
    and putting this back into \eqref{eq_proof_duallemma} we obtain that $A$ equals
    $$\lim_{l\to\infty}\E_{\underline{h}\in[N_l]^{k-1}}\E_{n\in[N_l]}
    g(n)\prod_{\underline{\eta} \in \{0,1\}^{k-1}_*} \mathcal{C}^{|\underline{\eta}|} \phi\big(n + \underline{\eta} \cdot \underline{h}\big)\lim_{H\to\infty}\E_{h_k\in[H]}
      \prod_{\underline{\eta} \in \{0,1\}^{k-1}} \mathcal{C}^{|\underline{\eta}|+1} \phi\big(h_k+ \underline{\eta} \cdot \underline{h}\big).$$

    Then using the Cauchy-Schwarz inequality,
    \begin{multline}
    \label{eq:oct-1-3}
        |A|^2
        \leq
         \lim_{l \to \infty} \E_{\underline{h} \in [N_l]^{k-1}}
         \left| \E_{n \in [N_l]} g(n) \prod_{\underline{\eta} \in \{0,1\}^{k-1}_*} \mathcal{C}^{|\underline{\eta}|} \phi(n + \underline{\eta} \cdot \underline{h}) \right|^2
        \times \\
        \lim_{l \to \infty} \E_{\underline{h} \in [N_l]^{k-1}} \left|\lim_{H\to\infty}\E_{h_k\in[H]} \prod_{\underline{\eta} \in \{0,1\}^{k-1}} \mathcal{C}^{|\underline{\eta}|+1} \phi(h_k+\underline{\eta} \cdot \underline{h})\right|^2.
    \end{multline}
We first deal with the second average of the right hand side of \eqref{eq:oct-1-3}.
    Observe that for every fixed $n\in\N$,
    $$\lim_{H\to\infty}\E_{h_k\in[H]} \prod_{\underline{\eta} \in \{0,1\}^{k-1}} \mathcal{C}^{|\underline{\eta}|} \phi(h_k+\underline{\eta} \cdot \underline{h})
    =
    \lim_{H\to\infty}\E_{h_k\in[H]} \prod_{\underline{\eta} \in \{0,1\}^{k-1}} \mathcal{C}^{|\underline{\eta}|} \phi(n+h_k+\underline{\eta} \cdot \underline{h}).$$
    Therefore, expanding the square, we have
\begin{multline}
\lim_{l \to \infty} \E_{\underline{h} \in [N_l]^{k-1}} \left|\lim_{H\to\infty}\E_{h_k\in[H]} \prod_{\underline{\eta} \in \{0,1\}^{k-1}} \mathcal{C}^{|\underline{\eta}|+1} \phi(h_k+\underline{\eta} \cdot \underline{h})\right|^2
\\
    \begin{split}
=~\lim_{l \to \infty} \E_{\underline{h} \in [N_l]^{k-1}} \Bigg( \lim_{N\to\infty}\E_{n \in[N]} \prod_{\underline{\eta} \in \{0,1\}^{k-1}} & \mathcal{C}^{|\underline{\eta}|+1} \phi(n+\underline{\eta} \cdot \underline{h}) \Bigg)
\\
    & \times\Bigg(\lim_{H\to\infty}\E_{h_k\in[H]} \prod_{\underline{\eta} \in \{0,1\}^{k-1}} \mathcal{C}^{|\underline{\eta}|} \phi(h_k+\underline{\eta} \cdot \underline{h}) \Bigg)
	\end{split}
\\
    \begin{split}
    =~\lim_{l \to \infty} \lim_{N\to\infty} \lim_{H\to\infty}\E_{\underline{h} \in [N_l]^{k-1}} \Bigg( \E_{n \in[N]} \prod_{\underline{\eta} \in \{0,1\}^{k-1}} & \mathcal{C}^{|\underline{\eta}|+1} \phi(n+\underline{\eta} \cdot \underline{h}) \Bigg)
\\
	\times\Bigg(\E_{h_k\in[H]} & \prod_{\underline{\eta} \in \{0,1\}^{k-1}} \mathcal{C}^{|\underline{\eta}|} \phi(n+h_k+\underline{\eta} \cdot \underline{h}) \Bigg)
	\end{split}
\\
    \begin{split}
=~\lim_{l \to \infty} \lim_{N\to\infty} \lim_{H\to\infty}\E_{\underline{h} \in [N_l]^{k-1}} \E_{n \in[N]} \E_{h_k\in[H]}\Bigg( \prod_{\underline{\eta} \in \{0,1\}^{k-1}} & \mathcal{C}^{|\underline{\eta}|+1} \phi(n+\underline{\eta} \cdot \underline{h}) \Bigg)
\\
    \times\Bigg( &\prod_{\underline{\eta} \in \{0,1\}^{k-1}} \mathcal{C}^{|\underline{\eta}|} \phi(n+h_k+\underline{\eta} \cdot \underline{h}) \Bigg)
	\end{split}
\\
   =~\lim_{l \to \infty} \E_{h_1, \ldots, h_{k-1} \in [N_l]} \lim_{N\to\infty} \E_{n\in[N]} \lim_{H\to\infty} \E_{h_k\in[H]}  \left( \prod_{\underline{\eta} \in \{0,1\}^k} \mathcal{C}^{|\underline{\eta}|} \phi(n + \underline{\eta} \cdot \underline{h}) \right).
\end{multline}
Suppose $\phi(n) = F(T^n x)$ for $n \in \N$ where $F$ is a continuous function in an ergodic nilsystem $(X, \mu, T)$. Then
\begin{multline}
\label{eq:jan-7-1}
    \lim_{l \to \infty} \E_{h_1, \ldots, h_{k-1} \in [N_l]} \lim_{N\to\infty} \E_{n\in[N]} \lim_{H\to\infty} \E_{h_k\in[H]} \prod_{\underline{\eta} \in \{0,1\}^k} \mathcal{C}^{|\underline{\eta}|} \phi(n + \underline{\eta} \cdot \underline{h}) = \\
    \lim_{l \to \infty} \E_{h_1, \ldots, h_{k-1} \in [N_l]} \lim_{N\to\infty} \E_{n\in[N]} \lim_{H\to\infty} \E_{h_k\in[H]}  \prod_{\underline{\eta} \in \{0,1\}^k} \mathcal{C}^{|\underline{\eta}|} T^{\underline{\eta} \cdot \underline{h}} F(T^n x) = \\
    \lim_{l \to \infty} \E_{h_1, \ldots, h_{k-1} \in [N_l]} \int_X  \lim_{H\to\infty} \E_{h_k\in[H]}  \prod_{\underline{\eta} \in \{0,1\}^k} \mathcal{C}^{|\underline{\eta}|} T^{\underline{\eta} \cdot \underline{h}} F d \mu
\end{multline}
where the last equality follows from the fact that the nilsystem $(X, \mu, T)$ is uniquely ergodic.
Since the limit inside the integral exists pointwise and $F$ is bounded, we can move that limit to the outside of the integral. Hence \eqref{eq:jan-7-1} is equal to
\begin{equation}
\label{eq:jan-9-1}
    \lim_{l \to \infty} \E_{h_1, \ldots, h_{k-1} \in [N_l]}   \lim_{H\to\infty} \E_{h_k\in[H]}  \int_X \prod_{\underline{\eta} \in \{0,1\}^k} \mathcal{C}^{|\underline{\eta}|} T^{\underline{\eta} \cdot \underline{h}} F d \mu = \hk{F}_{k}^{2^k}
\end{equation}
which is equal to $\lVert \phi \rVert_{U^k(\N)}^{2^k}$ by \cref{lem:two-norms-equal}.

    We now deal with the first average of the right hand side of \eqref{eq:oct-1-3}. Let $N'_l = kN_l$ and define $g_{N'_l}, \phi_{N'_l}: \Z_{N'_l} \to \C$ by $g_{N'_l} = g 1_{[N_l]}$ and $\phi_{N_l'} = \phi 1_{[N'_l]}$. Then
    \begin{multline*}
        \E_{\underline{h} \in [N_l]^{k-1}} \left| \E_{n \in [N_l]} g(n) \prod_{\underline{\eta} \in \{0,1\}^{k-1}_*} \mathcal{C}^{|\underline{\eta}+1|} \phi(n + \underline{\eta} \cdot \underline{h}) \right|^2
        \leq \\
        k^{k+1} \E_{\underline{h} \in \Z_{N_l'}^{k-1}} \left| \E_{n \in \Z_{N_l'}} g_{N'_l} (n) \prod_{\underline{\eta} \in \{0,1\}^{k-1}_*} \mathcal{C}^{|\underline{\eta}+1|} \phi_{N'_l}(n + \underline{\eta} \cdot \underline{h}) \right|^2.
    \end{multline*}
    Expanding the square and using the periodicity of $g_{N'_l}$ and $\phi_{N'_l}$, the right hand side of above inequality is equal to
    \begin{multline*}
        k^{k+1} \E_{\underline{h}\in \Z_{N'_l}^{k-1}} \bigg{[} \E_{n \in \Z_{N'_l}}\E_{h_k \in \Z_{N_l'}} g_{N'_l}(n) \prod_{\underline{\eta} \in \{0,1\}^{k-1}_*} \mathcal{C}^{|\underline{\eta}|+1} \phi(n + \underline{\eta} \cdot \underline{h}) \times \\
         \overline{g_{N_l'}}(n + h_k) \prod_{\underline{\eta} \in \{0,1\}^{k-1}_*} \mathcal{C}^{|\underline{\eta}|} \phi_{N_l'}(n + h_k+\underline{\eta} \cdot \underline{h}) \bigg{]} = \\
        k^{k+1} \E_{n \in \Z_{N_l'}} \E_{\underline{h} \in \Z_{N'_l}^k} g_{N_l'}(n) \overline{g_{N_l'}}(n + h_k) \prod_{\substack{\underline{\eta} \in \{0,1\}^k_* \\ \underline{\eta} \neq (0, 0, 0, \ldots, 0, 1)}} \mathcal{C}^{|\underline{\eta}|} \phi_{N_l'}(n + \underline{\eta} \cdot \underline{h}).
    \end{multline*}
    By Cauchy-Schwarz-Gowers inequality (\cref{lem:Cauchy_Schwarz_Gowers_inequality}), the right hand side of above equality is bounded by
    \begin{equation}
    \label{eq:oct-1-4}
        k^{k+1} \lVert g_{N_l'} \rVert_{U^k(\Z_{N_l'})}^2 \lVert \phi_{N_l'} \rVert_{U^k(\Z_{N_l'})}^{2^k - 2}.
    \end{equation}
    By definition, $\lVert g_{N_l'} \rVert_{U^k(\Z_{N_l'})} = \lVert g 1_{[N_l]} \rVert_{U^k(\Z_{kN_l})}$.
    On the other hand, according to  \cref{lem:sep-30-4},
    \begin{equation}
    \label{eq:apply_lem_sep_30-4}
        \lVert \phi_{N_l'} \rVert_{U^k(\Z_{N_l'})} \ll_{k} \lVert \phi \rVert_{U^k(\N)}
    \end{equation}
    for $N_l'$ sufficiently large.

    Combining \eqref{eq:oct-1-3}, \eqref{eq:jan-9-1}, \eqref{eq:oct-1-4} and \eqref{eq:apply_lem_sep_30-4}, we have the conclusion.
\end{proof}

The following theorem is the main ingredient in the proof of \cref{thm:decomposition_along_primes}.
\begin{theorem}
    \label{thm:decomposition_along_primes_part1}
    Given $k$ commuting measure preserving transformations $T_1,\dots,T_k$ on a probability space $(X,\B,\mu)$ and functions $f_0,\dots,f_k\in L^\infty(X)$, let
    \[
        \alpha(n)=\int_Xf_0\cdot T_1^n f_1\cdot T_2^n f_2\cdots T_k^n f_k\d\mu.
    \]
    Then for every $\epsilon > 0$ there exists a $k$-step nilsequence $\psi$ satisfying
    \begin{equation}
    \label{eq:july-10-5}
       \lim_{N\to\infty}\frac1{|\P\cap [N]|}\sum_{p\in \P\cap [N]}\big|\alpha(p)-\psi(p)\big| < \epsilon.
    \end{equation}
\end{theorem}

\begin{proof}
    Without loss of generality, assume that $\lVert f_i \rVert_{L^\infty} \leq 1$ for $0 \leq i \leq k$.
    First, we will show that for every nilsequence $\psi$, the limit in \eqref{eq:july-10-5} exists.
    Using \cite[Proposition 2.4]{Frantzikinakis15b}, we can uniformly approximate the nilsequence $\psi$ by a multicorrelation sequence.
    By \cite{Frantzikinakis_Host_Kra07, Frantzikinakis_Host_Kra13, Wooley_Ziegler12}, for every polynomial $P \in \Z[x]$ and multicorrelation sequence $\beta$, the limit
\[
    \lim_{N\to\infty}\frac1{|\P\cap [N]|}\sum_{p\in \P\cap [N]} P(\beta(p))
\]
exists. Hence, invoking the Stone-Weierstrass theorem, the limit
\[
    \lim_{N\to\infty}\frac1{|\P\cap [N]|}\sum_{p\in \P\cap [N]} |\beta(p)|
\]
also exists, and therefore so does the limit in \eqref{eq:july-10-5}.

    Fix $\epsilon > 0$ and choose $w$ so that it satisfies the conclusion of \cref{thm:dense_model_theorem} corresponding to $\epsilon$ and $k+1$.
    Let $W=\prod_{p<w}p$ and let $b \in [W]$ be coprime to $W$.
    The sequence $n\mapsto\alpha(Wn+b)$ is a $k$-multicorrelation sequence, so we can apply \cref{thm:strengthening_Frantzikinakis_decomposition_1} to get a $k$-step nilsequence $\psi_{W,b}$ with $\|\psi_{W,b}\|_\infty\leq1$ that is a convex combination of dual nilsequences $D_{k+1} \phi$ with $\lVert \phi \rVert_{U^{k+1}(\N)}\leq1$ and satisfies
    \begin{equation}
    \label{eq:oct-2-1}
        \lim_{N-M \to \infty} \E_{n \in [M, N)} |\alpha(Wn+b) - \psi_{W,b}(n)| \leq \epsilon.
   \end{equation}
    Keeping $W$ fixed, every $m\in\N$ can be written uniquely as $m=Wn+b$ for some $n\in\N$ and $b\in[W]$.
    Define the nilsequence $\psi=\psi_\epsilon$ as follows:
    \[
        \psi(m)=\psi(Wn +b) = \begin{cases}
        \psi_{W,b}(n) \mbox{ if } (b,W) = 1, \\
        0 \mbox{ if } (b,W) \neq 1.
        \end{cases}
    \]
    That $\psi$ is indeed an nilsequence follows from \cite[Lemma 2.1]{Leibman15}.
    In view of \cref{lem:von-Mangoldt},
    \begin{multline*}
        \E_{p \in \mathbb{P}} |\alpha(p) - \psi(p)|
    =
    \lim_{M \to \infty} \E_{m \in [M]} \Lambda'(m) |\alpha(m) - \psi(m)|
    \leq \\
    \E_{(b,W) = 1} \limsup_{N \to \infty} \E_{n \in [N]} \Lambda_{W,b}(n) |\alpha(Wn + b) - \psi_{W,b}(n)|.
    \end{multline*}
    In order to establish \eqref{eq:july-10-5} it suffices to show that for each $b\in[W]$ with $(b,W) = 1$,
    \begin{equation*}
        \limsup_{N \to \infty} \E_{n \in [N]} \Lambda_{W,b}(n) |\alpha(Wn + b) - \psi_{W,b}(n)| \ll \epsilon
    \end{equation*}
    By partitioning $[N]$ into $[N] = (N/3, N] \cup (N/9, N/3] \cup \ldots$, it suffices to show that
    \begin{equation}
    \label{eq:july-10-4}
        \limsup_{N \to \infty} \E_{n \in [N/4, 3N/4]} \Lambda_{W,b}(n) |\alpha(Wn + b) - \psi_{W,b}(n)| \ll \epsilon.
    \end{equation}
    The left hand side of \eqref{eq:july-10-4} can be written as
    \begin{equation}
    \label{eq:nov-22-7}
        \limsup_{N \to \infty} \E_{n \in [N/4, 3N/4]} g(n) \big(\alpha(Wn+b) - \psi_{W,b}(n)\big)
    \end{equation}
    {where\footnote{For a complex number $z\in\C$ we define $\sign(z)$ to be $z/|z|$ if $z\neq0$ and $\sign(0)=0$.} $g(n)=\Lambda_{W,b}(n)\cdot\sign\big(\alpha(Wn+b) - \psi_{W,b}(n)\big)$.}
    In particular, $|g(n)| \leq \Lambda_{W,b}(n)$ for $n \in \N$.

    We now use \cref{thm:dense_model_theorem} and the fact that $w$ (and hence $W$) was chosen to satisfy the conclusion of that corollary.
    Let $C=C(k)$ and $M=M(k)$ be the constants provided by that corollary.
    For each $N \in \N$, let $N'= C N$. We can decompose $g 1_{[N/4, 3N/4]} = g_{1, N'} + g_{2, N'}$ on $[N']$ in such a way that
    \begin{enumerate}
        \item $|g_{1, N'}(n)| \leq M \mbox{ for } n \in [N']$,
        \item $\lVert g_{2, N'} \rVert_{U^{k+1}(\Z_{N'})} \leq \epsilon,$
        \item and $g_{1,N'}, g_{2,N'}$ are supported on $[N]$.
    \end{enumerate}
    Here and throughout the proof, we allow the implicit constant in the notation $\ll$ to depend on $k$. Note that because
    \[
        \E_{n \in [N']} |g(n) 1_{[N/4, 3N/4]}(n)| \leq \E_{n \in [N']} \Lambda_{W, b}(n) \ll 1
    \]
    and $|g_{1,N'}(n)| \leq M$ pointwise, we have $\E_{n \in [N']} |g_{2, N'}(n)| \ll 1$.

    Then it follows that
    \begin{multline}
        \limsup_{N \to \infty} \E_{n \in [N/4, 3N/4]} g(n) \big(\alpha(Wn+b) - \psi_{W,b}(n)\big) \ll \\
        \limsup_{N' \to \infty} \E_{n \in [N']} g 1_{[N/4, 3N/4]} (n) \big(\alpha(Wn+b) - \psi_{W,b}(n)\big) = \\
        \limsup_{N' \to \infty} \E_{n \in [N']} \left( g_{1, N'} (n) + g_{2, N'} \right) \big(\alpha(Wn+b) - \psi_{W,b}(n)\big).
    \end{multline}
    By \eqref{eq:oct-2-1},
    \begin{multline}
    \label{eq:oct-2-2}
        \limsup_{N' \to \infty} \left| \E_{n \in N'} g_{1, N'} (n)  \big(\alpha(Wn+b) - \psi_{W,b}(n)\big) \right|
        \leq \\
        M \limsup_{N' \to \infty} \E_{n \in [N']} |\alpha(Wn+b) - \psi_{W,b}(n)|
        \ll
        \limsup_{N' \to \infty} \E_{n \in [N']} |\alpha(Wn+b) - \psi_{W,b}(n)|
        \ll
        \epsilon.
    \end{multline}
    On the other hand, according to \cite[Lemma 3]{Frantzikinakis_Host_Kra07} (or \cite[Lemma 3.5]{Frantzikinakis_Host_Kra13}),
    \begin{equation}
    \label{eq:oct-2-3}
        \limsup_{N' \to \infty} \left| \E_{n \in [N']} g_{2, N'}(n) \alpha(Wn+b) \right| \ll \limsup_{N' \to \infty} \lVert g_{2, N'} 1_{[N']} \rVert_{U^{k+1}(\Z_{(k+1)N'})}.
    \end{equation}
    Since $g_{2, N'}$ is supported on $[N]$ and $N' = C N \geq (k+1) N$,
    \begin{equation}
    \label{eq:jan-10-2}
        \limsup_{N' \to \infty} \lVert g_{2, N'} 1_{[N']} \rVert_{U^{k+1}(\Z_{(k+1)N'})} = \limsup_{N' \to \infty} \frac{1}{k+1} \lVert g_{2, N'} \rVert_{U^{k+1}(\Z_{N'})} \ll \epsilon.
    \end{equation}

Therefore, it remains to show that
    \begin{equation}
    \label{eq:nov-23-1}
        \limsup_{N' \to \infty} \left| \E_{n \in [N']} g_{2, N'}(n) \psi_{W,b}(n) \right| \ll \epsilon.
    \end{equation}
    Since $\psi_{W,b}$ is a convex combination of dual nilsequences of the form $D_{k+1} \phi$ with $\lVert \phi \rVert_{U^{k+1}(\N)} \leq 1$, it suffices to show that
    \[
        \limsup_{N' \to \infty} \left| \E_{n \in [N']} g_{2, N'}(n) D_{k+1} \phi(n) \right| \ll \epsilon
    \]
    for any nilsequence $\phi$ with $\lVert \phi \rVert_{U^{k+1}(\N)} \leq 1$. To this end, we will patch the $g_{2, N'}$ together to make use of \cref{lem:g_2_and_dual_function}. We will choose a fast growing sequence $(N'_l)_{l \in \N}$ of natural numbers and define $g_{2, \infty} : \N \to \C$ by
    \[
        g_{2, \infty} (n) = g_{2, N'_{l}}(n) \mbox{ for } n \in (N'_{l - 1}, N'_{l}].
    \]
    For $l \in \N$, $N_{l+1}'$ is picked very large compared to $N_l'$ so that we can ``identify" $g_{2, \infty}$ with $g_{2, N'_{l}}$ on $[N'_{l}]$. To be more precise, we need for every $l \in \N$,
    \[
        \E_{n \in [N'_{l}]} |g_{2, \infty} (n)| \leq  \E_{n \in [N'_{l}]} |g_{2, N'_{l}} (n)| +\epsilon\ll 1
    \]
    and
    \begin{equation}
    \label{eq:nov-22-2}
        \lVert g_{2, \infty} 1_{[N'_{l}]} \rVert_{U^{k+1}(\Z_{N'_{l}})} \leq  \lVert g_{2, N'_{l}} \rVert_{U^{k+1}(\Z_{N'_{l}})} +\epsilon \leq 2 \epsilon
    \end{equation}
    and
    \begin{equation}
    \label{eq:nov-22-5}
        \limsup_{l \to \infty} \left| \E_{n \in [N'_{l}]} g_{2, N'_{l}} (n) D_{k+1} \phi (n) \right| \leq  \limsup_{l \to \infty} \left| \E_{n \in [N'_{l}]} g_{2, \infty} (n) D_{k+1} \phi (n) \right|+\epsilon.
    \end{equation}
With the constructed $g_{2, \infty}$, applying \cref{lem:g_2_and_dual_function}, we have
    \begin{multline}
    \label{eq:nov-22-3}
        \limsup_{l \to \infty} \left| \E_{n \in [N'_l]} g_{2, \infty} (n) D_{k+1} \phi(n) \right|
        \ll
        \limsup_{l\to\infty}  \lVert g_{2, \infty} 1_{[N'_l]} \rVert_{U^{k+1}(\Z_{(k+1) N'_l})} \lVert \phi \rVert_{U^{k+1}(\N)}^{2^{k+1} -1}
        \\ \leq
        \limsup_{l\to\infty}  \lVert g_{2, \infty} 1_{[N'_l]} \rVert_{U^{k+1}(\Z_{(k+1) N'_l})}.
    \end{multline}
    Because $g_{2,\infty} 1_{[N_l']}$ is supported on $[N_l]$ where $N_l = N_l'/C \leq N_l'/(k+1)$, we have
    \begin{equation}
    \label{eq:jan-10-1}
        \lVert g_{2, \infty} 1_{[N_l']} \rVert_{U^{k+1}(\Z_{(k+1) N_l'})} = \frac{1}{k+1} \lVert g_{2, \infty} 1_{[N_l']} \rVert_{U^{k+1}(\Z_{N_l'})} \ll \epsilon.
    \end{equation}
    Combining \eqref{eq:nov-22-3} and \eqref{eq:jan-10-1},
    \[
        \limsup_{l \to \infty} \left| \E_{n \in [N'_l]} g_{2, \infty} (n) D_{k+1} \phi (n) \right| \ll \epsilon.
    \]
    Then by \eqref{eq:nov-22-5},
    \[
        \limsup_{l \to \infty} \left| \E_{n \in [N'_l]} g_{2, N'_l} (n) D_{k+1} \phi  \right| \ll \epsilon.
    \]
    This establishes \eqref{eq:nov-23-1}, and hence effectively shows that
\begin{equation}
\label{eq:nov-23-2}
    \E_{p \in \P} |\alpha(p) - \psi(p)| \ll \epsilon.
\end{equation}
\end{proof}

Now we are ready to prove \cref{thm:decomposition_along_primes}.
\begin{proof}[Proof of \cref{thm:decomposition_along_primes}]
    The main result from \cite{Frantzikinakis15b} guarantees that there exists a $k$-step nilsequence $\psi_0$ such that $\lVert \psi_0 \rVert_{\ell^{\infty}(\N)} \leq 1$ and
    \[
        \E_{n \in \N} |\alpha(n) - \psi_0(n)| < \epsilon/2.
    \]
    In view of \cref{thm:decomposition_along_primes_part1}, there exists a $k$-step nilsequence $\psi_1$ such that
    \[
        \E_{p \in \mathbb{P}} |\alpha(p) - \psi_1(p)| < \epsilon.
    \]
    Let $W$ be large enough so that $\phi(W)/W < \epsilon/8$ (such $W$ exists because $\lim_{W \to \infty} \phi(W)/W = 0$) and define the sequence $\psi$ as follows:
    \[
        \psi(n) = \begin{cases} \psi_0(n) \mbox{ if } (n,W) \neq 1, \\
                    \psi_1(n) \mbox{ if } (n,W) = 1.
        \end{cases}
    \]
    Then $\psi$ is a $k$-step nilsequence (see for example, \cite[Lemma 2.1]{Leibman15}). Because all but finitely many primes are coprime to $W$, we have
    \[
        \E_{p \in \mathbb{P}} |\alpha(p) - \psi(p)| = \E_{p \in \mathbb{P}} |\alpha(p) - \psi_{1}(p)| < \epsilon.
    \]
    On the other hand,

    \begin{eqnarray*}
      \E_{n \in \N} |\alpha(n) - \psi(n)|
      &=&
      \left( 1 - \frac{\phi(W)}{W} \right)\E_{(n,W)\neq1}|\alpha(n) - \psi(n)|
      +
      \frac{\phi(W)}{W} \E_{(n,W) = 1}|\alpha(n) - \psi(n)|
      \\&=&
      \E_{n\in\N}|\alpha(n) - \psi_0(n)|+\frac{\phi(W)}{W}\E_{(n,W)=1}|\alpha(n) - \psi_1(n)|-|\alpha(n) - \psi_0(n)|
      \\&\leq&
      \epsilon/2 + 4 \epsilon/8 = \epsilon.
    \end{eqnarray*}
\end{proof}

\begin{proof}[Proof of \cref{thm:decomposition_poly_correlation_along_primes}]
    The proof is almost identical to the proof of \cref{thm:decomposition_along_primes}. We explain the parts that need modifications.

    Let $\alpha$ be as defined in \eqref{eq:nov-27-1} and choose $W = \prod_{p \in \P, p < w}$ sufficiently large that satisfies the conclusion of \cref{thm:dense_model_theorem} corresponding to $\epsilon$ and $\ell+1$. By \cite[Theorem 1.2]{Frantzikinakis15b}, for every $r \in \N, s \in \Z$ and $b \in [W]$ with $(b,W) = 1$, the sequence $(\alpha(r(Wn + b) + s))_{n \in \N}$ is an approximate $\ell$-step nilsequence. Moreover, they are $\ell$-antiuniform with anti-uniform seminorm bounded by $1$ (see \cite[Proposition 7 Section 23.2]{Host_Kra_18} or \cite[Proposition 6.1]{Frantzikinakis_Host_2017}). In view of \cref{thm:strengthening_Frantzikinakis_decomposition}, the sequence $(\alpha(r(Wn + b) + s))_{n \in \N}$ can be approximated in $\ell^2(\N)$ by convex combinations of dual nilsequences of the form $D_{\ell + 1} \phi$ with $\lVert \phi \rVert_{U^{\ell+1}(\N)} \leq 1$. Proceeding as in the proof of \cref{thm:decomposition_along_primes_part1}, we can find an $\ell$-step nilsequence $\psi_1$ with $\lVert \psi_1 \rVert_{\infty} \leq \lVert \alpha \rVert_{\infty}$ such that
    \[
        \E_{p \in \P} |\alpha(rp + s) - \psi_1(rp + s)| \leq \epsilon.
    \]
    By \cite{Frantzikinakis15b}, there exists an $\ell$-step nilsequence $\psi_0$ with $\lVert \psi_0 \rVert_{\infty} \leq \lVert \alpha \rVert_{\infty}$ such that
    \[
        \E_{n \in \N} |\alpha(n) - \psi_1(n)| \leq \epsilon.
    \]
    Gluing $\psi_0$ and $\psi_1$ together as in the proof of \cref{thm:decomposition_along_primes}, we obtain a nilsequence $\psi$ satisfying the conclusion of \cref{thm:decomposition_poly_correlation_along_primes}.
\end{proof}

\section{Open questions}
\label{sec:open_questions}

\begin{question}
\label{question:1}
    Let $\alpha$ be as defined in \eqref{eq:oct-14-1}. Is it true that for any $\epsilon > 0$ and finite collection of non-constant polynomials $Q_1, Q_2, \ldots, Q_t \in \Z[x]$, there exists a nilsequence $\psi$ such that for all $1 \leq i \leq t$,
    \[
        \lim_{N - M \to \infty} \frac{1}{N-M} \sum_{n=M}^{N-1} \left| \alpha(Q_i(n)) - \psi(Q_i(n)) \right| \leq \epsilon
    \]
    and
    \[
        \lim_{N\to\infty}\frac1{|\P\cap [1,N]|}\sum_{p\in \P\cap [1,N]}\big|\alpha(Q_i(p))-\psi(Q_i(p))\big| \leq \epsilon?
    \]
\end{question}

Note that, in view of \cref{thm:decomposition_poly_correlation_along_primes}, the answer to this question is affirmative for $t=1$. 

We can ask a similar question for Hardy field sequences.

\begin{question}
    Let $\alpha$ be as defined in \eqref{eq:oct-14-1}. Is it true that for any $\epsilon > 0$ and $c > 0$, there exists a $k$-step nilsequence $\psi$ such that
    \[
        \lim_{N - M \to \infty} \frac{1}{N-M} \sum_{n=M}^{N-1} |\alpha(n) - \psi(n)| \leq \epsilon
    \]
    and
    \[
       \limsup_{N\to\infty}\frac1{N}\sum_{n \in [N]}\big|\alpha(\lfloor n^c \rfloor)-\psi(\lfloor n^c \rfloor)\big| \leq \epsilon.
    \]
    where $\lfloor x \rfloor$ denotes integer part of $x$.
\end{question}

The following question has been asked several times in the literature, see \cite[Remark after Theorem 1.1]{Frantzikinakis15b}, \cite[Problem 20]{Frantzikinakis_17}, \cite[Problem 1, Section 2.7]{Frantzikinakis_Host_2018}, and \cite[Page 398]{Host_Kra_18}.

\begin{question}\label{question_zeroepsilon}
  Let $\alpha$ be defined as in \eqref{eq:oct-14-1} (or, more generally, as in \eqref{eq:nov-27-1}).
  Does there exist a uniform limit of nilsequences $\psi$ such that
  $$\lim_{N-M\to\infty}\frac1{N-M}\sum_{n=M}^N\big|\alpha(n)-\psi(n)\big|=0?$$
\end{question}

If the answer to \cref{question_zeroepsilon} is affirmative, then the method in \cite{Tao_Teravainen_17} can be used to answer affirmatively the following seemingly more difficult question:

\begin{question}\label{question_zeroepsilon_prime}
   Let $\alpha$ be defined as in \eqref{eq:oct-14-1} (or, more generally, as in \eqref{eq:nov-27-1}).
  Does there exist a uniform limit of nilsequences $\psi$ such that for all $r \in \N, s \in \N \cup\{0\}$,
  $$\lim_{N-M\to\infty}\frac1{N-M}\sum_{n=M}^N\big|\alpha(rn + s)-\psi(rn +s)\big|=0$$
  and
  \[
    \lim_{N \to \infty} \frac{1}{|\P \cap [N]|} \sum_{p \in \P \cap [N]} \big| \alpha(rp + s) - \psi(rp+s)\big| = 0?
  \]
\end{question}


\section*{Acknowledgments} 
We thank Bryna Kra and Nikos Frantzikinakis for helpful feedback and comments. The third author is supported by the National Science Foundation under grant number DMS 1901453. We also thank the anonymous referee for useful suggestions. 

\bibliographystyle{amsplain}

\begin{thebibliography}{10}

\bibitem{Auslander_Green_Hahn63}
L.~Auslander, L.~Green, and F.~Hahn.
\newblock {\em Flows on homogeneous spaces}.
\newblock With the assistance of L. Markus and W. Massey, and an appendix by L.
  Greenberg. Annals of Mathematics Studies, No. 53. Princeton University Press,
  Princeton, N.J., 1963.

\bibitem{Bergelson_Host_Kra05}
V.~Bergelson, B.~Host, and B.~Kra.
\newblock Multiple recurrence and nilsequences.
\newblock {\em Invent. Math.}, 160(2):261--303, 2005.
\newblock With an appendix by I. Ruzsa.

\bibitem{Frantzikinakis15b}
N.~Frantzikinakis.
\newblock Multiple correlation sequences and nilsequences.
\newblock {\em Invent. Math.}, 202(2):875--892, 2015.

\bibitem{Frantzikinakis_17}
N.~Frantzikinakis.
\newblock Some open problems on multiple ergodic averages.
\newblock {\em Bull. Hellenic Math. Soc.}, 60:41--90, 2016.

\bibitem{Frantzikinakis_Host_2017}
N.~Frantzikinakis and B.~Host.
\newblock Higher order {F}ourier analysis of multiplicative functions and
  applications.
\newblock {\em J. Amer. Math. Soc.}, 30(1):67--157, 2017.

\bibitem{Frantzikinakis_Host_2018}
N.~Frantzikinakis and B.~Host.
\newblock Weighted multiple ergodic averages and correlation sequences.
\newblock {\em Ergodic Theory Dynam. Systems}, 38(1):81--142, 2018.

\bibitem{Frantzikinakis_Host_Kra07}
N.~Frantzikinakis, B.~Host, and B.~Kra.
\newblock Multiple recurrence and convergence for sequences related to the
  prime numbers.
\newblock {\em J. Reine Angew. Math.}, 611:131--144, 2007.

\bibitem{Frantzikinakis_Host_Kra13}
N.~Frantzikinakis, B.~Host, and B.~Kra.
\newblock The polynomial multidimensional {S}zemer\'{e}di theorem along shifted
  primes.
\newblock {\em Israel J. Math.}, 194(1):331--348, 2013.

\bibitem{Furstenberg81}
H.~Furstenberg.
\newblock {\em Recurrence in ergodic theory and combinatorial number theory}.
\newblock Princeton University Press, Princeton, N.J., 1981.

\bibitem{Gowers01}
W. T.~Gowers.
\newblock A new proof of Szemer\'edi's theorem.
\newblock {\em GAFA}, 11:465--588, 2001.

\bibitem{Green_Tao08}
B.~Green and T.~Tao.
\newblock The primes contain arbitrarily long arithmetic progressions.
\newblock {\em Annals of Math.}, 167:481--547, 2008.

\bibitem{Green_Tao10}
B.~Green and T.~Tao.
\newblock Linear equations in primes.
\newblock {\em Ann. of Math. (2)}, 171(3):1753--1850, 2010.

\bibitem{Green_Tao12}
B.~Green and T.~Tao.
\newblock The {M}\"obius function is strongly orthogonal to nilsequences.
\newblock {\em Ann. of Math. (2)}, 175(2):541--566, 2012.

\bibitem{Green_Tao_Ziegler12}
B.~Green, T.~Tao, and T.~Ziegler.
\newblock An inverse theorem for the {G}owers {$U^{s+1}[N]$}-norm.
\newblock {\em Ann. of Math. (2)}, 176(2):1231--1372, 2012.

\bibitem{Host_Kra05}
B.~Host and B.~Kra.
\newblock Nonconventional ergodic averages and nilmanifolds.
\newblock {\em Ann. of Math. (2)}, 161(1):397--488, 2005.

\bibitem{Host_Kra_18}
B.~Host and B.~Kra.
\newblock {\em Nilpotent structures in ergodic theory}, volume 235 of {\em
  Mathematical Surveys and Monographs}.
\newblock American Mathematical Society, 2018.

\bibitem{Host_Kra09}
B. Host and B. Kra.
\newblock Uniformity seminorms on {$\ell^\infty$} and applications.
\newblock {\em J. Anal. Math.}, 108:219--276, 2009.

\bibitem{Le17}
A.~Le.
\newblock Nilsequences and multiple correlations along subsequences.
\newblock To appear in Ergodic Theory Dynam. Systems.
  https://doi.org/10.1017/etds.2018.110, 2017.

\bibitem{Leibman10}
A.~Leibman.
\newblock Multiple polynomial correlation sequences and nilsequences.
\newblock {\em Ergodic Theory Dynam. Systems}, 30(3):841--854, 2010.

\bibitem{Leibman15}
A.~Leibman.
\newblock Nilsequences, null-sequences, and multiple correlation sequences.
\newblock {\em Ergodic Theory Dynam. Systems}, 35(1):176--191, 2015.

\bibitem{Parry_1973}
W.~Parry.
\newblock Dynamical representations in nilmanifolds.
\newblock {\em Compositio Math.}, 26:159--174, 1973.

\bibitem{Tao_Teravainen_17}
T.~Tao and J.~Ter{\"a}v{\"a}inen.
\newblock The structure of logarithmically averaged correlations of
  multiplicative functions, with applications to the {C}howla and {E}lliott
  conjectures.
\newblock {\em To appear in Duke Math. J.}, 2017.
\newblock Available at arXiv:1708.02610.

\bibitem{Walsh12}
M.~Walsh.
\newblock Norm convergence of nilpotent ergodic averages.
\newblock {\em Ann. of Math. (2)}, 175(3):1667--1688, 2012.

\bibitem{Wooley_Ziegler12}
T.~Wooley and T.~Ziegler.
\newblock Multiple recurrence and convergence along the primes.
\newblock {\em Amer. J. Math.}, 134(6):1705--1732, 2012.

\end{thebibliography}


\begin{dajauthors}
\begin{authorinfo}[pgom]
    Anh N.\ Le\\
    Ohio State University\\\par\nopagebreak
    Columbus, OH, USA\\
\href{mailto:le.286@osu.edu}
{\texttt{le.286@osu.edu}}
\end{authorinfo}
\begin{authorinfo}[johan]
    Joel Moreira\\
    University of Warwick\\
    Coventry, UK\\
\href{mailto:joel.moreira@warwick.ac.uk}
{\texttt{joel.moreira@warwick.ac.uk}}
\end{authorinfo}

\begin{authorinfo}[laci]
    Florian K.\ Richter\\
    Northwestern University\\
    Evanston, IL, USA\\
\href{mailto:fkr@northwestern.edu}
{\texttt{fkr@northwestern.edu}}
\end{authorinfo}

\end{dajauthors}

\end{document}